\documentclass[12pt]{amsart}
\usepackage{amsmath, amsthm, amssymb, tikz, color, hyperref}
\usepackage[normalem]{ulem}
\usepackage[initials]{amsrefs}
\usepackage[margin=1.5in]{geometry}

\def\papermode{1}

\newcommand{\wreathproduct}{lexicographic product}
\newcommand{\directproduct}{full product}
\newcommand{\ageindivisible}{indivisible}
\newcommand{\ageindivisibility}{indivisibility}
\newcommand{\AgeIndivisibility}{Indivisibility}
\newcommand{\Ageindivisibility}{Indivisibility}

\newtheorem{theorem}{Theorem}[section]
\newtheorem{corollary}[theorem]{Corollary}
\newtheorem{lemma}[theorem]{Lemma}
\newtheorem{proposition}[theorem]{Proposition}

\newtheorem{question}[theorem]{Question}
\theoremstyle{remark}
\newtheorem{remark}[theorem]{Remark}
\newtheorem{example}[theorem]{Example}
\theoremstyle{definition}
\newtheorem{definition}[theorem]{Definition}

\newcommand{\fraisse}{Fra\"{i}ss\'{e}}
\newcommand{\KK}{\mathbf{K} }
\newcommand{\CC}{\mathcal{C} }
\newcommand{\I}{\mathcal{I}}
\newcommand{\M}{\mathcal{M}}
\newcommand{\N}{\mathcal{N}}

\newcommand{\C}{\mathcal{C}}
\newcommand{\LL}{\mathcal{L}}
\newcommand{\Kp}{\mathcal{K}}
\newcommand{\E}{\mathcal{E}}
\newcommand{\oa}{\overline{a} }
\newcommand{\ob}{\overline{b} }
\newcommand{\oc}{\overline{c} }
\newcommand{\od}{\overline{d} }
\newcommand{\ox}{\overline{x} }
\newcommand{\oy}{\overline{y} }
\newcommand{\oz}{\overline{z} }
\newcommand{\btimes}{\boxtimes }
\newcommand{\mathiff}{\text{iff }}
\newcommand{\wreath}{\wr }
\newcommand{\qfequiv}[1]{\equiv^{\mathrm{qf}}_{#1}}
\newcommand{\qfclass}[3]{[#1]^{#2}_{#3}}

\DeclareMathOperator{\tp}{tp}
\DeclareMathOperator{\qftp}{qftp}

\DeclareMathOperator{\sig}{sig}
\DeclareMathOperator{\arity}{arity}
\DeclareMathOperator{\age}{age}

\DeclareMathOperator{\Flim}{Flim}
\DeclareMathOperator{\emb}{emb}
\DeclareMathOperator{\Aut}{Aut}
\DeclareMathOperator{\Pow}{\mathcal{P}}
\DeclareMathOperator{\Fml}{Fml}
\newcommand{\GG}{\mathbf{G} }
\newcommand{\HH}{\mathbf{H} }
\newcommand{\PG}{\mathbf{PG} }
\newcommand{\LO}{\mathbf{LO} }
\newcommand{\EE}{\mathbf{E} }
\newcommand{\PO}{\mathbf{PO} }
\newcommand{\TT}{\mathbf{T} }
\newcommand{\FF}{\mathbf{F} }
\newcommand{\DG}{\mathbf{DG} }

\title{Products of Classes of Finite Structures}
\author[V Guingona]{Vincent Guingona}
\address{Towson University, 7800 York Rd., Towson, MD, 21252}
\email{vguingona@towson.edu}
\author[M Parnes]{Miriam Parnes}
\address{Towson University, 7800 York Rd., Towson, MD, 21252}
\email{mparnes@towson.edu}
\author[L Scow]{Lynn Scow}
\address{California State University, San Bernardino, 5500 University Parkway, San Bernardino, CA 92407}
\email{lscow@csusb.edu}

\begin{document}

\begin{abstract}
	We study the preservation of certain properties under products of classes of finite structures.  In particular, we examine indivisibility, definable self-similarity, the amalgamation property, and the disjoint $n$-amalgamation property.  We explore how each of these properties interacts with the \wreathproduct, \directproduct, and free superposition of classes of structures.  Additionally, we consider the classes of theories which admit configurations indexed by these products.  In particular, we show that, under mild assumptions, the products considered in this paper do not yield new classes of theories.
\end{abstract}

\maketitle

\section{Introduction}\label{Section_Introduction}

In this paper, we consider classes of finite structures in a relational language $L$, and closed under $L$-isomorphism.  We study when certain partition properties are preserved under product operations on pairs of classes.  As an example of such a partition property, a class $\KK$ of finite $L$-structures is \emph{indivisible} if for all $A \in \KK$, there exists $B \in \KK$ such that for all 2-colorings of the elements in $B$, there exists an $L$-isomorphic copy $A'$ of $A$ in $B$, such that all elements of $A'$ are colored the same color (see Definition \ref{Definition_AgeIndivis}).  This is the usual definition of age indivisibility for a countable structure $\Gamma$ when $\KK$ is the age of $\Gamma$, but we study it in a more general context.  It was shown in \cite{ezs} that age indivisibility of $\Gamma$ follows from \emph{indivisibility} of $\Gamma$, which latter property holds when any finite partition of $\Gamma$ contains an isomorphic copy of $\Gamma$ in one of the pieces of the partition.  It is not too hard to see that the rational order and the Rado graph are indivisible.  In \cite{Henson}, the homogeneous $K_n$-free graphs $H_{2,n}$ (for $n \geq 3$) were shown to be age indivisible, and they were shown to be indivisible in \cite{KR86} and \cite{ezs89}.

In \cite{GuiPar} the notion of definable self-similarity was introduced for infinite structures $\Gamma$.  In the case that $\Gamma$ is a homogeneous relational structure and $\age(\Gamma)$ has the strong amalgamation property, $\Gamma$ is \emph{definably self-similar} if and only if for any finite set $A \subseteq \Gamma$ and $x \notin A$, the orbit of $x$ under the group of automorphisms of $\Gamma$ fixing $A$ pointwise is isomorphic to $\Gamma$, which latter property was shown in A.4.9 of \cite{FrSa00} to imply that $\Gamma$ is indivisible.  Though the Rado graph is definably self-similar, the $H_{2,n}$ (for $n \geq 3$) are not, as worked out in the paragraph after Fact 1.2 of \cite{sauer2020}.  In Definition \ref{Definition_DefinablySelfSim}, we define a more general notion of definable self-similarity for a class $\KK$ of finite structures, which restricts to the previous notion in the case that $\KK$ is a \fraisse~class of structures in a relational language with the strong amalgamation property.  In Definition \ref{Definition_DisjNAmalg}, we adapt a notion of disjoint $n$-amalgamation from \cite{Kru} and prove, in Proposition \ref{Proposition_DisAmalgDefSelfSim}, that a class $\KK$ of finite structures in a relational language with the hereditary property and exactly one singleton structure, up to isomorphism, is definably self-similar if it has disjoint 3-amalgamation.

Given two classes of finite structures $\KK_0$ and $\KK_1$, there are several ways to form a ``product'' of these classes, as described in Definition \ref{Definition_Wreath_Classes}: the free superposition $\KK_0 * \KK_1$, the full product $\KK_0 \btimes \KK_1$, and the lexicographic product $\KK_0 \wreath \KK_1$.  The free superposition of classes was defined in \cite{Cam90} and later studied in \cite{bod14} and \cite{bod15} in the context of the Ramsey property (see Definition \ref{Definition_RP}).  The full product of structures $\Gamma_0 \btimes \Gamma_1$ (see Definition \ref{Definition_Direct}) is an adjunct to this study in the latter two papers, and we propose what we consider to be the appropriate generalization to classes, as evidenced by Lemma \ref{Lemma_AgeProducts}.  The lexicographic product of classes is based on the lexicographic product of structures studied in \cite{Meir16}, \cite{Meir22} and \cite{Tou} (see Definition \ref{Definition_Wreath} and the discussion thereafter).  The preservation of strong amalgamation by free superpositions was known (see Proposition \ref{Proposition_Known}) and we point out some further \fraisse-type properties preserved by lexicographic and full products in Proposition \ref{Proposition_WreathDirectAP}.

Our results on when indivisibility is preserved by lexicographic and full products of classes are contained in Theorems \ref{Theorem_WreathAgeInd} and \ref{Theorem_DirectAgeInd}, respectively, however the question for free superposition remains open (see Question \ref{Question_FreeAgeInd}).  Of the three ``special'' products that we study, only free superposition preserves the property of being definably self-similar (see Subsection \ref{Subsection_DSS}).  Our results on when disjoint $n$-amalgamation is preserved by lexicographic products and free superposition of classes are contained in Theorems \ref{Theorem_WreathStrongAmalg} and \ref{Theorem_FreenAmalgamation}, respectively, with counterexamples for the full product given in the same subsection.  A table summarizing the known results on preservation of partition properties by products of classes is given in the Conclusion.

The motivation for this paper comes from model theory, in particular, the study of ``dividing lines" in classification theory, where each complete first-order theory $T$ falls on either the ``more-complicated'' or the ``less-complicated'' side of any one dividing line.  The property of being stable is one such dividing line given in \cite{sh78}, building on previous work from, e.g., \cite{Morley} (see the Historical Remarks in \cite{sh78} for more details.)  A theory $T$ is \emph{stable} if there is no formula $\varphi(\ox;\oy)$ that defines a linear order relation on an infinite set of parameters $(\oa_i : i<\omega)$ from a model of $T$.  A theory $T$ does not have the \emph{independence property} (i.e., is \emph{NIP}) if there is no formula $\varphi(\ox;\oy)$ that defines a random (Rado) graph relation on an infinite set of parameters from a model of $T$.  The existence of a formula $\varphi(\ox;\oy)$ and parameters that witness certain positive and negative instances of $\varphi(\ox;\oy)$ in some model, such as the two properties defined above, may be informally referred to as the existence of a ``definable configuration''. 

For certain definable configurations, it can be useful to define an ordinal-valued rank which measures the extent to which we can achieve the configuration within a partial type $\pi(\ox)$.  Some examples of these ranks are dp-rank and burden which measure the extent to which inp- and ict-patterns (defined in \cite{sh78}) can be achieved within the type, respectively.  In \cite{KOU}, additivity of dp-rank was shown, and in \cite{Adler}, it was proved that in NIP theories, the dp-rank of a type is equal to its burden. In \cite{Chernikov}, indiscernible sequences were used to show the sub-multiplicativity of burden. The first two authors of this paper proposed a notion of $\KK$-rank for a general class $\KK$ in \cite{GuiPar}, and asked the natural question of when $\KK$-rank is additive.  If $\KK$ is the class of finite linear orders, $\KK$-rank coincides with op-dimension in NIP theories (Proposition 5.9 of \cite{GuiPar}) and if $\KK$ is the class of finite equivalence relations, the $\KK$-rank of the partial type $\oy=\oy$ coincides with its dp-rank.  In many of the previous results, it was the property of being definably self-similar or indivisible that allowed $\KK$-ranks to be so well understood (see Propositions 4.13 and 4.14 of \cite{GuiPar}) which motivates our study, in this paper, of the operations on classes that preserve these desirable properties.  

In \cite{sh78}, $I$-indexed (generalized) indiscernible sequences were introduced to better understand dividing lines (see Definition \ref{Definition_GenInd}).  In \cite{sh78}, it was shown that a theory $T$ is stable if and only if any indiscernible sequence in any model $M \vDash T$ is an indiscernible set.  This result was generalized in \cite{Scow2012} to the case of NIP theories, with the observation that $M \vDash T$ ``forces a reduction on the language of the index model.''  This was a starting point for the work in \cite{GuiHillScow}, where a precise definition is given of the class of theories $T$ characterized by a certain indiscernible ``collapse''.  One limitation of the ``collapse'' approach is that the Ramsey property for the class $\KK = \age(I)$ was used in an essential way to show that certain definable configurations can be captured by certain $I$-indexed indiscernible sequences, as shown in \cite{Scow2012}.  This was one motivation for the definition of ``$T$ admits a $\KK$-configuration'' (see Definition \ref{Definition_Configuration}) which generalizes the property ``$T$ admits a non-collapsing $\Flim(\KK)$-indexed indiscernible'', in the case that $\KK$ has the Ramsey property.  The notion of $\KK$-configuration was first developed in \cite{GuiHill} and later generalized in \cite{GuiPar}.  Let $\CC_\KK$ denote the class of complete theories that admit a $\KK$-configuration.  Certain classes of theories on the ``complicated'' side of the dividing line are exactly captured as $\CC_\KK$, for certain classes $\KK$ (see Remark \ref{Remark_DividingLines}).

Our incomplete understanding of the classes $\CC_\KK$ motivated us to attempt to find classes $\KK$ such that $\CC_\KK$ had not been previously studied.  Though we investigate some new classes of finite structures in a binary signature, we do not end up introducing any previously unknown classes $\CC_\KK$ in this way.  Proposition \ref{Proposition_POTTDGGG} determines that $\CC_\KK$, where $\KK$ is either the class of finite partial orders, the class of finite tournaments, or the class of finite directed graphs, coincides with $\CC_{\GG}$, the class of complete theories with IP.  In Section \ref{Section_Configurations}, we pose Question \ref{Question_Classes}, some form of which was stated in \cite{GuiHill} and \cite{GuiPar}, which essentially asks if the $\CC_\KK$ are linearly ordered under inclusion, for any class $\KK$ of finite structures in a relational signature.  With the techniques developed in this paper, we were not able to construct a counterexample to the linear order.  Namely, Theorem \ref{Theorem_Classes} states that, under mild assumptions on $\KK_0$ and $\KK_1$, if $\KK'$ is one of our special products of the classes $\KK_0$ and $\KK_1$, then $\C_{\KK'}$ is equal to the class $\CC_{\KK_0} \cap \CC_{\KK_1}$.  This result inspires some further questions that we state in Section \ref{Section_Conclusion}.

\section{Preliminaries}\label{Section_Preliminaries}

Until we talk about configurations in Section \ref{Section_Configurations}, we will work in a relational language $L$.  Let $\sig(L)$ denote the set of non-logical symbols in $L$.  We will consider both $L$-structures $A$ and classes of finite $L$-structures $\KK$ that are closed under isomorphism.  If $A$ and $B$ are $L$-structures, let $\emb_L(A,B)$ denote the set of all $L$-embeddings from $A$ to $B$.  Let $\Aut_L(A)$ denote the set of all $L$-automorphisms of $A$.  We drop the $L$ on these if the language is understood.  If $A$ is an $L$-structure, we write $B \subseteq A$ to denote that $B$ is an $L$-substructure of $A$.  We will conflate an $L$-structure $A$ with its universe (instead of writing, for example, $|A|$).  If $L$ is a language, $M$ is an $L$-structure, and $A \subseteq M$, let $\Fml(L,A)$ denote the set of all $L$-formulas with parameters in $A$.

\begin{definition}[\fraisse~Classes]\label{Definition_Fraisse}
	Let $L$ be a relational language and let $\KK$ be a class of finite $L$-structures closed under isomorphism.
	\begin{enumerate}
		\item We say that $\KK$ has the \emph{hereditary property} if, for all $A \in \KK$ and $B \subseteq A$, $B \in \KK$.
		\item We say that $\KK$ has the \emph{joint embedding property} if, for all $A, B \in \KK$, there exists $C \in \KK$, $f \in \emb_L(A,C)$, and $g \in \emb_L(B,C)$.
		\item We say that $\KK$ has the \emph{amalgamation property} if, for all $A, B_0, B_1 \in \KK$, $f_0 \in \emb_L(A,B_0)$, and $f_1 \in \emb_L(A,B_1)$, there exists $C \in \KK$, $g_0 \in \emb_L(B_0, C)$, and $g_1 \in \emb_L(B_1, C)$ such that $g_0 \circ f_0 = g_1 \circ f_1$.  If, furthermore, we can choose this such that $g_0(B_0) \cap g_1(B_1) = g_0(f_0(A))$, we say $\KK$ has the \emph{strong amalgamation property} (or the \emph{disjoint amalgamation property}).
		\item We say that $\KK$ is a \emph{\fraisse~class} if it is non-empty, it has countably many isomorphism classes of structures, it has the hereditary property, the joint embedding property, and the amalgamation property.
	\end{enumerate}
\end{definition}

Throughout this paper, we will be interested primarily in classes that have at least the hereditary property.  We will see that it helps to have the joint embedding property for certain results to hold, but typically amalgamation is not necessary.

\begin{example}\label{Example_BasicExamples}
	Consider the following examples, which will be referenced throughout the paper, in their natural languages.
	\begin{enumerate}
		\item Let $\GG$ be the class of all finite graphs.  Then, $\GG$ is a \fraisse~class.
		\item For each $k \ge 2$, let $\HH_k$ be the class of all finite $k$-uniform hypergraphs.  Then, $\HH_k$ is a \fraisse~class.
		\item Let $\LO$ be the class of all finite linear orders.  Then, $\LO$ is a \fraisse~class.
		\item Let $\EE$ be the class of all finite sets equipped with a single equivalence relation.  Then, $\EE$ is a \fraisse~class.
		\item Let $\PG$ be the subclass of $\GG$ that contains only planar graphs (graphs that can be embedded into $\mathbb{R}^2$).  Then, $\PG$ has the hereditary property and the joint embedding property, but it does not have the amalgamation property.
		\item Let $\FF$ be the subclass of $\GG$ that contains graphs with no cycles (forests).  Then, $\FF$ has the hereditary property and the joint embedding property, but not the amalgamation property.
		\item Let $\PO$ be the class of all finite partial orders.  Then, $\PO$ is a \fraisse~class.
		\item Let $\TT$ be the class of all finite tournaments (directed graphs obtained by assigning a direction to each undirected edge of a complete graph).  Then, $\TT$ is a \fraisse~class.
	\end{enumerate}
\end{example}

\begin{proof}
    Most of this can be found in Chapter 7 of \cite{BigHodges}.  We prove the more obscure ones here.  For the proofs below, suppose $R$ is the unique (binary) relation symbol in the language.

    (5): We show that $\PG$ does not have the amalgamation property.	Define elements of $\PG$ with underlying universes $A = 4$, $B_0 = 6$, and $B_1 = 5$, where $A$ has no $R$-edges,
	\[
		B_0 \models \bigwedge_{i<3} R(i,4) \wedge R(i,5), \text{ and } B_1 \models \bigwedge_{i<4} R(i,4).
	\]
	Let $f_0$ and $f_1$ be the identity embeddings of $A$ into $B_0$ and $B_1$ respectively.  Then, it is clear that any completion of this amalgamation induces a $K_{3,3}$, which does not belong to $\PG$ \cite{Kur}.
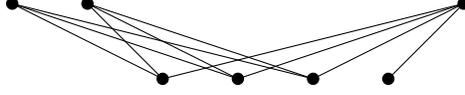
	\begin{center}
		\begin{tikzpicture}
			\foreach \x in {2,3,4,5} {
				\draw (6,1) -- (\x,0);
				\filldraw (\x,0) circle (0.075);
			}
			\filldraw (6,1) circle (0.075);
			\foreach \x in {0,1} {
				\foreach \y in {2,3,4}
					\draw (\x,1) -- (\y,0);
				\filldraw (\x,1) circle (0.075);
			}
		\end{tikzpicture}
	\end{center}

    (6): We show that $\FF$ does not have the amalagamtion property.  Let $A = 2$, $B_0 = 3$, and $B_1 = 4$ be elements of $\FF$, where $A$ has no $R$-edges,
	\[
		B_0 \models R(0,2) \wedge R(2,1), \text{ and } B_1 \models R(0,2) \wedge R(2,3) \wedge R(3,1).
	\]
	Let $f_0$ and $f_1$ be the identity embeddings of $A$ into $B_0$ and $B_1$ respectively.  Then, it is clear that any completion of this amalgamation contains a $5$-cycle, which does not belong to $\FF$.
\end{proof}

Below we define the age of a structure and note how it interacts with the properties in Definition \ref{Definition_Fraisse}.

\begin{definition}[Age]
	Let $L$ be a relational language and let $M$ be an $L$-structure.  The \emph{age} of $M$, denoted $\age(M)$, is the class of all finite $L$-structures $A$ so that $\emb_L(A,M) \neq \emptyset$.
\end{definition}

\begin{theorem}[Theorem 7.1.1 of \cite{BigHodges}]
	Let $L$ be a relational language and let $\KK$ be a class of finite $L$-structures closed under isomorphism that has countably many isomorphism classes of structures.  Then, $\KK$ has the hereditary property and the joint embedding property if and only if there exists an $L$-structure $M$ such that $\KK = \age(M)$.
\end{theorem}

\begin{definition}[Homogeneity]
	Let $L$ be a relational language and let $M$ be an $L$-structure.  We say that $M$ is \emph{homogeneous} if, for all finite $A \subseteq M$ and for all $f \in \emb_L(A,M)$, there exists $g \in \Aut_L(M)$ extending $f$.  Note that this notion is sometimes called \emph{ultrahomogeneous}, and should not be confused with ``$\kappa$-\emph{homogeneous}'' as in, e.g., Definition 4.2.12 of \cite{Marker}.
\end{definition}

\begin{theorem}[\fraisse's Theorem, \cite{Fr54}]
	Let $L$ be a relational language and let $\KK$ be a class of finite $L$-structures closed under isomorphism.  Then, $\KK$ is a \fraisse~class if and only if there exists a unique (up to isomorphism) countable structure $M$ such that $M$ is homogeneous and $\KK = \age(M)$.
\end{theorem}

When the conditions of \fraisse's Theorem are met, we call $M$ the \emph{\fraisse~limit} of $\KK$, and denote it by $M = \Flim(\KK)$.

\begin{example}
	The \fraisse~limit of $\GG$ is the random (Rado) graph.  The \fraisse~limit of $\LO$ is $(\mathbb{Q}; <)$.
\end{example}

\subsection{Partition Properties}

We repeat a definition from \cite{Ne05} followed by a related definition from \cite{ezs}.

\begin{definition}[Ramsey Property for Embeddings / Objects]\label{Definition_RP}
	For $L$-structures $A$ and $C$, let $\binom{C}{A}$ denote \emph{$A$-subobjects of $C$} where this can be defined either to be the set of embeddings of $A$ into $C$, $\emb_L(A,C)$, or to be $\emb_L(A,C)/\sim$, where for $f, g \in \emb_L(A,C)$, $f\sim g$ if and only if there exists $h \in \Aut_L(A)$ such that $g=f \circ h$.  In the following, we can identify the element $f \in \binom{C}{A}$, such that $f: A \rightarrow A'$ with its range, the structure $A'$.

	Given a fixed notion of subobject, we say that a class of finite $L$-structures $\KK$ has the \emph{Ramsey property} if for all $A, B \in \KK$ and for any integer $k \geq 2$, there exists $C \in \KK$ such that for any $c : \binom{C}{A} \rightarrow k$, there is a structure $B' \in \binom{C}{B}$ such that for any $A', A'' \in \binom{B'}{A}$, $c(A') = c(A'')$.  If we define $\binom{C}{A} :=\emb_L(A,C)$, then this property is referred to as the \emph{Ramsey property for embeddings}; if we define $\binom{C}{A}:=\emb_L(A,C)/\sim$, then this property is referred to as the \emph{Ramsey property for objects}.  In the case that all structures are rigid (i.e., have no non-trivial automorphisms), the two properties are equivalent.
\end{definition}

\begin{definition}[\AgeIndivisibility]\label{Definition_AgeIndivis}
	Let $L$ be a relational language and let $\KK$ be a class of finite $L$-structures closed under isomorphism.  We say that $\KK$ is \emph{\ageindivisible} if, for all $A \in \KK$ and $k \ge 2$, there exists $B \in \KK$ such that, for all $c : B \rightarrow k$, there exists $f \in \emb_L(A,B)$ such that $|c(f(A))| = 1$.
\end{definition}

\begin{example}
	The classes $\GG$, $\HH_k$ for $k \ge 2$, $\LO$, $\EE$, $\PO$, and $\TT$ are \ageindivisible.  However, $\PG$ is not \ageindivisible\ (because of the Four Color Theorem, for example) and $\FF$ is not \ageindivisible\ (by alternating colors).
\end{example}

When $\KK$ has only one singleton structure up to isomorphism, then the Ramsey property (for embeddings or for objects) implies indivisibility.

The notion we call ``indivisible'' here is called ``age indivisible'' in Definition 2.9 of \cite{GuiPar}.  We have changed the name of this property to align with other sources, such as \cite{ezs}, since an arbitrary class of finite structures is not necessarily the age of some structure.  In the setting where $\KK$ is the age of some countable structure, the notion of indivisibility from Definition \ref{Definition_AgeIndivis} is equivalent to the usual notion of \emph{age indivisibility} (for structures) in \cite{ezs}.

\begin{definition}[Age Indivisibility]
	Let $L$ be a relational language and $M$ be a countable $L$-structure.  We say $M$ is \emph{age indivisible} if, for all $k < \omega$, for all $c : M \rightarrow k$, there exists $i < k$ such that $\age(c^{-1}(\{ i \})) = \age(M)$.
\end{definition}

\begin{lemma}[Theorem 1 of \cite{ezs}]\label{Lemma_OldAgeIndivisible}
    Suppose that $L$ is a relational language and $M$ is a countable $L$-structure.  Then, $\age(M)$ is \ageindivisible\ if and only if $M$ is age indivisible.
\end{lemma}

\begin{definition}[Indivisibility]
	Let $L$ be a relational language and let $M$ be an $L$-structure.  We say $M$ is \emph{indivisible} if, for all $k \ge 2$ and $c : M \rightarrow k$, there exists $N \subseteq M$ such that $|c(N)| = 1$ and $N$ is $L$-isomorphic to $M$.
\end{definition}

By Lemma \ref{Lemma_OldAgeIndivisible}, we see that the indivisibility of $M$ implies \ageindivisibility\ of $\age(M)$.

\begin{corollary}[Theorem 1 of \cite{ezs}]
	Let $L$ be a relational language and let $M$ be a countable $L$-structure.  If $M$ is indivisible, then $\age(M)$ is \ageindivisible.
\end{corollary}
\
The converse does not hold (i.e., if $\age(M)$ is indivisible, it does not follow that $M$ is indivisible).

\begin{example}\label{Example_AgeIndNotInd}
	Let $L$ be the language with two binary relation symbols, $E_0$ and $E_1$, and let $M$ be the $L$-structure with universe $\omega \times \omega$ where, for all $a, b, c, d < \omega$,
	\begin{align*}
		M \models E_0((a,b), (c,d)) & \Longleftrightarrow a = c, \text{ and} \\
		M \models E_1((a,b), (c,d)) & \Longleftrightarrow b = d.
	\end{align*}
	Then, it is not hard to check that $\age(M)$ is \ageindivisible\ (see, for example, Theorem \ref{Theorem_DirectAgeInd}).  However, $M$ is not indivisible.  For example, consider the coloring $c : M \rightarrow 2$ given by
	\[
		c(a,b) = \begin{cases} 1 & \text{ if } a < b \\ 0 & \text{ if } a \ge b \end{cases}.
	\]
	If $N \subseteq M$ and $N \cong M$, fix $(a,b) \in N$.  If $a < b$, then there exist only finitely many $d \in \omega$ with $c(d,b) = c(a,b) = 1$ and, similarly, if $a \ge b$, then there exist only finitely many $d \in \omega$ with $c(a,d) = c(a,b) = 0$.  Thus, $|c(N)| > 1$.
\end{example}

These notions are related to the notion of definable self-similarity from \cite{GuiPar}. 
 Let $L$ be a relational language, $A$ an $L$-structure, $A_0 \subseteq A$, and $a, a' \in A$.  Write $a \qfequiv{A_0} a'$ if the function sending $a$ to $a'$ and fixing $A_0$ is an $L$-isomorphism from $\{ a \} \cup A_0$ to $\{ a' \} \cup A_0$ (i.e., $a$ and $a'$ have the same quanitifier-free $L$-type over $A_0$).  This is clearly an equivalence relation.  Write $\qfclass{a}{A}{A_0}$ to denote the $\qfequiv{A_0}$-class of $a$ in $A$.  Since $L$ is relational, $\qfclass{a}{A}{A_0}$ is a substructure of $A$.

\begin{definition}[Definably Self-Similar]\label{Definition_DefinablySelfSim}
	Let $L$ be a relational language and let $\KK$ be a class of finite $L$-structures closed under isomorphism.  We say $\KK$ is \emph{definably self-similar} if, for all $A, B, C \in \KK$, for all $f \in \emb_L(A,B)$, for all $C_0 \subseteq C$, for all $c \in C \setminus C_0$, and for all $g \in \emb_L(A, \qfclass{c}{C}{C_0})$, there exists $D \in \KK$, $j \in \emb_L(C,D)$, and $h \in \emb_L(B, \qfclass{j(c)}{D}{j(C_0)})$ such that $h \circ f = j \circ g$.
\end{definition}

The definition of definably self-similar given here differs from the original definition in \cite{GuiPar}, but it is equivalent for \fraisse~classes via Lemma 3.14 of \cite{GuiPar}.  We use this version of the definition here because it is better suited for classes of structures.

\begin{lemma}[A.4.9 of \cite{FrSa00}, Lemma 3.15 of \cite{GuiPar}]\label{Lemma_DefSelfSimAgeInd}
	Let $L$ be a relational language and let $\KK$ be a class of finite $L$-structures closed under isomorphism.  If $\KK$ is a \fraisse~class and $\KK$ is definably self-similar, then $\Flim(\KK)$ is indivisible.  Therefore, $\KK$ is \ageindivisible.
\end{lemma}

%

\subsection{Disjoint $n$-Amalgamation}

We can generalize the strong amalgamation property to higher dimensions, following Section 3 of \cite{Kru}, suitably modified for classes of structures.  For $n \ge 2$, let $\Pow(n)$ be the power set of $n$, which is a partially ordered set under inclusion.

\begin{definition}
	Let $n \ge 2$, let $P \subseteq \Pow(n)$ closed under intersections, let $L$ be a relational language, and let $\KK$ be a class of finite $L$-structures closed under isomorphism.  A \emph{disjoint $P$-amalgamation system in $\KK$} is
	\begin{itemize}
		\item $A_p \in \KK$ for all $p \in P$, and
		\item $f_{p,q} \in \emb(A_p, A_q)$ for all $p, q \in P$ with $p \subseteq q$
	\end{itemize}
	such that
	\begin{itemize}
		\item (identity) for all $p \in P$, $f_{p,p}$ is the identity embedding,
		\item (commutivity) for all $p, q, r \in P$, if $p \subseteq q \subseteq r$, then
			\[
				f_{p,r} = f_{q,r} \circ f_{p,q},
			\]
			and
		\item (disjointness) for all $p, q, r \in P$ with $p, q \subseteq r$,
			\[
				f_{p,r}(A_p) \cap f_{q,r}(A_q) = f_{p \cap q,r}( A_{p \cap q} ).
			\]
	\end{itemize}
\end{definition}

If $P' \subseteq P \subseteq \Pow(n)$ and $\mathcal{P} = ((A_p)_{p \in P}, (f_{p,q})_{p,q \in P, p \subseteq q})$ is a disjoint $P$-amalgamation system in $\KK$, then $\mathcal{P}' = ((A_p)_{p \in P'}, (f_{p,q})_{p,q \in P', p \subseteq q})$ is clearly a disjoint $P'$-amalgamation system in $\KK$.  We say that $\mathcal{P}$ \emph{extends} $\mathcal{P}'$.

\begin{definition}\label{Definition_DisjNAmalg}
	Let $L$ be a relational language and let $\KK$ be a class of finite $L$-structures closed under isomorphism.  For $n \ge 2$, we say that $\KK$ has the \emph{disjoint $n$-amalgamation property} if every disjoint ($\Pow(n) \setminus \{ n \}$)-amalgamation system in $\KK$ can be extended to a disjoint $\Pow(n)$-amalgamation system.
\end{definition}

The following is the commuting diagram for disjoint $3$-amalgamation:
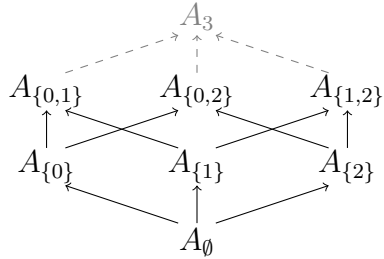
\begin{center}
	\begin{tikzpicture}
		\draw (0,0) node {$A_\emptyset$};
		\foreach \x in {0,1,2} {
			\draw ({2*(\x-1)},1) node {$A_{\{ \x \}}$};
			\draw[->] ({0.25*(\x-1)},0.25) -- ({1.75*(\x-1)},0.75);
		}
		\draw (-2,2) node {$A_{\{0,1\}}$};
		\draw (0,2) node {$A_{\{0,2\}}$};
		\draw (2,2) node {$A_{\{1,2\}}$};
		\foreach \x in {-1,1} {
			\draw[->] ({\x*2},1.25) -- ({\x*2},1.75);
			\draw[->] ({\x*0.25},1.25) -- ({\x*1.75},1.75);
			\draw[->] ({\x*1.75},1.25) -- ({\x*0.25},1.75);
		}
		\foreach \x in {-1,0,1}
			\draw[color = gray, dashed, ->] ({1.75*\x},2.25) -- ({0.25*\x},2.75);
		\draw[color = gray] (0,3) node {$A_3$};
	\end{tikzpicture}
\end{center}

Note that disjoint $2$-amalgamation is equivalent to strong amalgamation.  See Example 3.3 of \cite{Kru} for an explanation of the connection between the definition given in this paper and the one from \cite{Kru}.

\begin{remark}
	If $\KK$ has the disjoint $n$-amalgamation property for some fixed $n \ge 2$, then it has the disjoint $m$-amalgamation property for all $2 \le m \le n$.  This is because we can convert a disjoint $\Pow(m)$-amalgamation system into a disjoint $\Pow(n)$-amalgamation system by adding trivial inclusions.
\end{remark}

\begin{lemma}\label{Lemma_TransitiveNo3Amalg}
	Let $L$ be a relational language with a binary relation symbol, $E$.  Let $\KK$ be a class of finite $L$-structures such that
    \begin{enumerate}
        \item $\KK$ is closed under isomorphism,
        \item $\KK$ has the hereditary property,
        \item there exists only one singleton up to isomorphism in $\KK$,
        \item for all $A \in \KK$, $E$ is transitive on $A$, and
        \item there exists $A_0, A_1 \in \KK$, each with universe $\{0, 1\}$, such that $A_0 \models E(0,1)$ and $A_1 \models \neg E(0,1)$.
    \end{enumerate}
    Then, $\KK$ does not have the disjoint $3$-amalgamation property.
\end{lemma}

\begin{proof}
	For each $X \subset 3$, let $A_X$ be an $L$-structure with universe $X$ and suppose that
	\[
		A_{\{0,1\}} \models E(0,1), \ A_{\{1,2\}} \models E(1,2), \text{ and } A_{\{0,2\}} \models \neg E(0,2).
	\]
	Let $f_{\{i\},\{i,j\}}(i) = i$ for all $i, j < 3$.  One can check that this is a disjoint $(\Pow(3) \setminus \{ 3 \})$-amalgamation system in $\KK$.  Then, if there were an extension to a disjoint $\Pow(3)$-amalgamation system in $\KK$, then $A_3$ would contain $a_0, a_1, a_2$ with
	\[
		A_3 \models E(a_0, a_1) \wedge E(a_1, a_2) \wedge \neg E(a_0, a_2),
	\]
	contrary to the fact that $E$ is transitive on $A_3$.  Thus, no such system exists, and $\KK$ does not have the disjoint $3$-amalgamation property.
\end{proof}

\begin{example}
	For all $k \ge 2$ and all $n \ge 2$, $\HH_k$ has the disjoint $n$-amalgamation property (see Example 3.7 of \cite{Kru}).  Moreover, $\TT$ has the disjoint $n$-amalgamation property for all $n \ge 2$ (clear).  The classes $\LO$, $\EE$, and $\PO$ have the disjoint $2$-amalgamation property (i.e., strong amalgamation), but, by the previous lemma, these classes do not have the disjoint $n$-amalgamation property for $n \ge 3$.
\end{example}

Under mild assumptions, having the disjoint $3$-amalgamation property implies being definably self-similar.  The authors were made aware through personal communication that R. Patel has proved stronger partition results that imply Proposition \ref{Proposition_DisAmalgDefSelfSim} in the case that $\KK$ is the age of some structure.

\begin{proposition}\label{Proposition_DisAmalgDefSelfSim}
	Let $L$ be a relational language and let $\KK$ be a class of finite $L$-structures closed under isomorphism.  If $\KK$ has exactly one singleton structure up to isomorphism, $\KK$ has the hereditary property, and $\KK$ has the disjoint $3$-amalgamation property, then $\KK$ is definably self-similar.
\end{proposition}

\begin{proof}
	Suppose that $A, B, C \in \KK$, $f \in \emb(A,B)$, $C_0 \subseteq C$, $c \in C \setminus C_0$, and $g \in \emb(A, \qfclass{c}{C}{C_0})$.  Without loss of generality, we may assume $|B| = |A|+1$.  Let $A_\emptyset = \emptyset$, $A_{\{0\}} = A$, $A_{\{1\}}$ be a singleton, and $A_{\{2\}} = C_0$.  Let $A_{\{0,1\}} = B$ and $A_{\{0,2\}} = A_{\{1,2\}} = C$.  Let $f_{\{2\},\{0,2\}}$, and $f_{\{2\},\{1,2\}}$ be the identity embeddings, let $f_{\{0\}, \{0,1\}} = f$, let $f_{\{0\},\{0,2\}} = g$, let $f_{\{1\}, \{0,1\}}$ be any embedding of $A_{\{1\}}$ into $B \setminus g(A)$, and let $f_{\{1\}, \{1,2\}}$ be any embedding of $A_{\{1\}}$ into $\qfclass{c}{C}{C_0}$ (these exist since $\KK$ contains exactly one singleton up to isomorphism).  It is easy to check that this is a disjoint ($\Pow(3) \setminus \{ 3 \}$)-amalgamation system.  Since $\KK$ has disjoint $3$-amalgamation, this extends to a disjoint $\Pow(3)$-amalgamation system.  Then, let $D = A_3$, let $j = f_{\{0,2\},3}$, and let $h = f_{\{0,1\},3}$.  Clearly $h(b) \in \qfclass{j(c)}{D}{j(C_0)}$ and $h \circ f = j \circ g$.  Therefore, $\KK$ is definably self-similar.

	\begin{center}
		\begin{tikzpicture}
			\draw (0,0) node {$\emptyset$};
			\draw (-2,1) node {$A$};
			\draw (0,1) node {$A_{\{1\}}$};
			\draw (2,1) node {$C_0$};
			\foreach \x in {0,1,2} {
				\draw[->] ({0.25*(\x-1)},0.25) -- ({1.75*(\x-1)},0.75);
			}
			\draw (-2,2) node {$B$};
			\draw (0,2) node {$C$};
			\draw (2,2) node {$C$};
			\foreach \x in {-1,1} {
				\draw[->] ({\x*2},1.25) -- ({\x*2},1.75);
				\draw[->] ({\x*0.25},1.25) -- ({\x*1.75},1.75);
				\draw[->] ({\x*1.75},1.25) -- ({\x*0.25},1.75);
			}
			\draw (-2,1.5) node[anchor = east] {$f$};
			\draw (-0.75,1.6) node[anchor = south] {$g$};
			\foreach \x in {-1,0,1}
				\draw[color = gray, dashed, ->] ({1.75*\x},2.25) -- ({0.25*\x},2.75);
			\draw[color = gray] (0,2.4) node[anchor = west] {$j$};
			\draw[color = gray] (-1,2.5) node[anchor = south] {$h$};
			\draw[color = gray] (0,3) node {$D$};
		\end{tikzpicture}
	\end{center}
\end{proof}

\begin{example}
	The converse of Proposition \ref{Proposition_DisAmalgDefSelfSim} is in general false.  For example, consider $\LO$ the class of all finite linear orders.  There is only one singleton up to isomorphism and $\LO$ is definably self-similar.  However, $\LO$ does not have the disjoint $3$-amalgamation property.
\end{example}

\begin{example}
	The assumption of having exactly one singleton structure in $\KK$ up to isomorphism in Proposition \ref{Proposition_DisAmalgDefSelfSim} is necessary.  For example, let $L$ be the language with a single unary predicate $P$ and let $\KK$ be the class of all finite $L$-structures.  Then, $\KK$ has the disjoint $3$-amalgamation property (and in fact the disjoint $n$-amalgamation property for all $n \ge 2$), but $\KK$ is not definably self-similar.
	\if\papermode0
		Let $p(x)$ be the type $P(x)$ over $\emptyset$, let $A$ be a singleton structure with a realization of $P$, and let $B$ be a two element structure with one realization of $P$ and one realization of $\neg P$.  Then, though there is an embedding of $A$ into the realizations of $p$, there is no embedding of $B$.
	\fi
\end{example}

\subsection{Products of Structures}\label{Subsection_Structures}

In this subsection, we will define various products of structures.  To begin, we need two relational languages, $L_0$ and $L_1$.

\begin{definition}[Lexicographic Product]\label{Definition_Wreath}
	Let $L_0$ and $L_1$ be relational languages, let $B$ be an $L_1$-structure, and, for each $b \in B$, let $A_b$ be an $L_0$-structure.  Let $L_2$ be the language whose signature is the disjoint union of the signatures of $L_0$ and $L_1$ together with a new binary relation symbol, $E$.  Then, the \emph{\wreathproduct\ of $A_b$ along $B$}, denoted $\bigsqcup_{b \in B} A_b$, is the $L_2$-structure with universe
	\[
		\{ (a,b) : b \in B, a \in A_b \}
	\]
	such that, for all $R \in \sig(L_2)$ of arity $n$, for all $b_0, \dots, b_{n-1} \in B$, for all $a_0 \in A_{b_0}, \dots, a_{n-1} \in A_{b_{n-1}}$,
	\[
		\bigsqcup_{b \in B} A_b \models R((a_0, b_0), \dots, (a_{n-1}, b_{n-1}))
	\]
	if and only if
	\begin{enumerate}
		\item if $R \in \sig(L_0)$, $b_0 = \dots = b_{n-1}$ and $A_{b_0} \models R(a_0, \dots, a_{n-1})$; and 
		\item if $R \in \sig(L_1)$, $B \models R(b_0, \dots, b_{n-1})$; and
		\item if $R = E$ (and $n = 2$), $b_0 = b_1$.
	\end{enumerate}
	If $A_b = A$ for all $b \in B$, then we let
	\[
		A \wreath B = \bigsqcup_{b \in B} A_b.
	\]
	This is called the \emph{\wreathproduct\ of $A$ and $B$}.
\end{definition}

\begin{remark}\label{Remark_WreathHKO}
    The definition of lexicographic product here is essentially Definition 2.1 of \cite{Tou}, which is based on Definition 1.12 of \cite{Meir16} and Definition 1.1 of \cite{Meir22}.  What we call $\bigsqcup_{b \in B} A_b$ is interdefinable with $B[\{ A_b : b \in B \}^s]^{\mathbb{U}}$ from \cite{Tou} (there $A_b$ and $B$ are in the same language, but this is a minor obstruction).  Both of these are technically different from $B[A_b]^s_{b \in B}$ in \cites{Meir16, Meir22}; for example, if $b \in B$, and $a_0, a_1 \in A_b$ are such that $B \models E(b,b)$ and $A \not\models E(a_0, a_1)$, then $B \models E(b,b)$ but $B[A_b]^s_{b \in B} \not\models E((a_0,b), (a_1,b))$.  Note that this also comes up, for example, in \cite{HKO}, where it is called ``composition.''
\end{remark}

\begin{remark}\label{Remark_ScowSokic} In Definition 5.6 of \cite{sc21}, the third author defines ``semi-direct product structures'' according to the group construction in \cite{kpt05}.  In summary, given relational signatures $L_1,L_2$ that intersect on an ordering $\{\prec\}$, an $L_2$-structure $\N$ ordered by $\prec$ and $L_1$-structures $(\M_i : i \in \N)$ ordered by $\prec$, we create a convex ordering on $\I_{i \in \N}(\M_i)$, where each $\M_i$ is defined to be an equivalence class under some new equivalence relation $E$.  In Theorem 5.13 of \cite{sc21} it is proved that if $\age(\N)$ has the Ramsey property and there is a class $\Kp$ such that $\age(\M_i) = \Kp$ for all $i \in \N$ such that $\Kp$ has the Ramsey property, then the age of $\I_{i \in \N}(\M_i)$ has the Ramsey property.  

Up to some decisions about where the ordering lies in the languages, the age of $\I_{i \in \N}(\M_i)$ is exactly $\C[\LL^*]\E[\Kp^*]$ as in the work of M.~Soki\'{c}, and Theorem 4.4 of \cite{sokQuotients} proves that if $\LL^*$ has JEP, then $\Kp^*$ and $\LL^*$ have the Ramsey property if and only if $\C[\LL^*]\E[\Kp^*]$ has the Ramsey property.

Since $\Kp$, $\age(\N)$ are assumed to be ordered Ramsey classes, they must have the amalgamation property, and so the results from \cite{kpt05} were known to provide an alternate proof of the transfer of the Ramsey property to the semi-direct product structure.
However, the third author was unaware of the finitary argument in \cite{sokQuotients} at the time of publication of \cite{sc21}, so we mention it here.
\end{remark}

The notation in Definition \ref{Definition_Wreath} is used because of the relationship between the automorphism group of $A \wreath B$ and the automorphism groups of $A$ and $B$.

Let $G$ and $H$ be groups and let $G$ act on a set $B$.  The \emph{wreath product} of $H$ and $G$, denoted $H \wreath G$, is the group with underlying set $\left( \prod_B H \right) \times G$ and group operation given by, for all $g, g' \in G$ and $h_b, h'_b \in H$ for $b \in B$,
\[
	((h_b)_{b \in B}, g) \cdot ((h'_b)_{b \in B}, g') = ((h_b \cdot h'_{g^{-1} \cdot b})_{b \in B}, g \cdot g').
\]

\begin{proposition}[Lemma 2.13 of \cite{Meir22}]\label{Proposition_WreathAut}
	Let $L_0$ and $L_1$ be relational languages, let $A$ be an $L_0$-structure, and let $B$ be an $L_1$-structure.  Then,
	\[
		\Aut(A \wreath B) \cong \Aut(A) \wreath \Aut(B).
	\]
\end{proposition}

Note that, although the statement of Lemma 2.13 of \cite{Meir22} assumes the transitivity of $B$, this is not needed in the proof if we use our definition of lexicographic product (it is only used in the first bullet point in the converse direction of the proof, which is unnecessary in our setting).

\begin{definition}[Full Product]\label{Definition_Direct}
	Let $L_0$ and $L_1$ be relational languages, let $A$ be an $L_0$-structure, and let $B$ be an $L_1$-structure.  Let $L_2$ be the language whose signature is the disjoint union of the signatures of $L_0$ and $L_1$ together with two new binary relation symbols, $E_0$ and $E_1$.  The \emph{\directproduct\ of $A$ and $B$}, denoted $A \btimes B$, is the $L_2$-structure with universe $A \times B$ such that, for all $R \in \sig(L_2)$ of arity $n$, for all $a_0, \dots, a_{n-1} \in A$, for all $b_0, \dots, b_{n-1} \in B$,
	\[
		A \btimes B \models R((a_0, b_0), \dots, (a_{n-1}, b_{n-1}))
	\]
	if and only if
	\begin{enumerate}
		\item if $R \in \sig(L_0)$, $A \models R(a_0, \dots, a_{n-1})$; and
		\item if $R \in \sig(L_1)$, $B \models R(b_0, \dots, b_{n-1})$; and
		\item if $R = E_0$, $a_0 = a_1$; and
		\item if $R = E_1$, $b_0 = b_1$.
	\end{enumerate}
\end{definition}

\begin{remark}
    This is precisely Definition 3.2 of \cite{bod15}, which differs slightly from the definition in Section 3 of \cite{bod14}.
\end{remark}

There is an obvious relationship between the automorphism groups of $A \btimes B$, $A$, and $B$.

\begin{proposition}[Proposition 3.1 of \cite{bod14}]
	Let $L_0$ and $L_1$ be relational languages, let $A$ be an $L_0$-structure, and let $B$ be an $L_1$-structure.  Then,
	\[
		\Aut(A \btimes B) \cong \Aut(A) \times \Aut(B).
	\]
\end{proposition}

\begin{lemma}[\cite{bod14}]\label{Lemma_DirectEmbeddings}
	Let $L_0$ and $L_1$ be relational languages, let $A$ and $C$ be $L_0$-structures, and let $B$ and $D$ be $L_1$-structures.  For all $f \in \emb\left( A \btimes B, C \btimes D \right)$, there exists $g \in \emb(A, C)$ and $h \in \emb(B, D)$ such that, for all $a \in A$ and $b \in B$,
	\[
		f(a,b) = (g(a), h(b)).
	\]
	Moreover, for any such $g$ and $h$, constructing $f$ in this manner gives an embedding from $A \btimes B$ to $C \btimes D$.
\end{lemma}

\if\papermode0
	\begin{proof}
		Fix $f \in \emb\left( A \btimes B, C \btimes D \right)$.  For any $a \in A$ and $b, b' \in B$, $\pi_0(f(a,b)) = \pi_0(f(a,b'))$.  Thus, we can define $g : A \rightarrow C$ by setting, for all $a \in A$, $g(a) = \pi_0(f(a,b))$ for some (any) choice of $b \in B$.  One can easily check that $g \in \emb(A,C)$.  Construct $h \in \emb(B,D)$ similarly.
	\end{proof}
\fi

Indivisibility interacts nicely with \wreathproduct s but not with \directproduct s.

\begin{proposition}[Proposition 2.21 of \cite{Meir16}]\label{Proposition_WreathIndivisible}
	Let $L_0$ and $L_1$ be relational languages, let $A$ be an $L_0$-structure, and let $B$ be an $L_1$-structure.  If $A$ and $B$ are indivisible, then so is $A \wreath B$.
\end{proposition}

Note that, even though the proof of Proposition 2.21 of \cite{Meir16} uses Meir's definition of lexicographic product, indivisibility implies that it is equivalent to our definition.

The analogue of Proposition \ref{Proposition_WreathIndivisible} does not hold for \directproduct s.

\begin{example}\label{Example_DirectIndivisible}
	Let $L$ be the language with an empty signature and let $A$ be the $L$-structure that is countably infinite.  Then, $A$ is indivisible (by the Pigeonhole Principle).  However, $A \btimes A$ is not indivisible.  In fact, $A \btimes A$ is isomorphic to the structure from Example \ref{Example_AgeIndNotInd}.
\end{example}

\begin{proposition}[Proposition 2.14(2) of \cite{Meir22}, Proposition 3.3 of \cite{bod15}]\label{Proposition_Homogeneity}
	Let $L_0$ and $L_1$ be relational languages, let $A$ be an $L_0$-structure, and let $B$ be an $L_1$-structure.  If $A$ and $B$ are homogeneous, then both $A \wreath B$ and $A \btimes B$ are homogeneous.
\end{proposition}

What we are calling ``homogeneity'' here is what is called ``strongly $\aleph_0$-homogeneous'' in Proposition 2.14 of \cite{Meir22}.  Note that the proof of Proposition 3.3 of \cite{bod15} does not use the assumption of $\aleph_0$-categoricity nor the assumption of being Ramsey to prove the transfer of homogeneity.

\section{Products of Classes}\label{Section_Classes}

In this section, we are interested in extending the definitions of products from Subsection \ref{Subsection_Structures} to classes of structures.

\begin{definition}\label{Definition_Wreath_Classes}
	For each $t < 2$, let $L_t$ be a relational language and let $\KK_t$ be a class of finite $L_t$-structures closed under isomorphism.
	\begin{enumerate}
		\item Let $\KK_0 \wreath \KK_1$ be the class of all structures isomorphic to $\bigsqcup_{b \in B} A_b$ for some $B \in \KK_1$ and $A_b \in \KK_0$ for each $b \in B$.  This is called the \emph{\wreathproduct} of $\KK_0$ and $\KK_1$.
		\item Let $\KK_0 \btimes \KK_1$ be the class of all structures embeddable into $A \btimes B$ for some $A \in \KK_0$ and $B \in \KK_1$.  This is called the \emph{\directproduct} of $\KK_0$ and $\KK_1$.
		\item Let $L_2$ be the language whose signature is the disjoint union of the signatures of $L_0$ and $L_1$.  Let $\KK_0 * \KK_1$ be the class of all $L_2$-structures $A$ such that $A |_{L_0} \in \KK_0$ and $A |_{L_1} \in \KK_1$.  This is called the \emph{free superposition} of $\KK_0$ and $\KK_1$.
	\end{enumerate}
\end{definition}

\begin{remark}
	For brevity, when using these definitions in proofs, we will forgo writing ``up to isomorphism.''  That is, we will assume that a structure in $\KK_0 \btimes \KK_1$ is a substructure of $A \btimes B$ for $A \in \KK_0$ and $B \in \KK_1$ (instead of merely isomorphic to such a structure) and similarly for $\KK_0 \wreath \KK_1$.
\end{remark}

\begin{remark}\label{Remark_SuperpositionStructures}
	One could, in general, define the free superposition of two structures, $A$ and $B$ given an accompanying bijection from $A$ to $B$.  Let $f : A \rightarrow B$ be a bijection from $A$ to $B$ and define $A *_f B$ to be a structure in the language whose signature is the disjoint union of the signatures of the language of $A$ and the language of $B$.  The universe of $A *_f B$ is $A$ and we superimpose the structure of $B$ on it via $f$.  Note that this is the same (up to interdefinability) as taking the $f$-diagonal of $A \btimes B$.
\end{remark}

This is, in some respect, the ``correct'' generalization of products to classes of structures, as the age commutes with the \wreathproduct\ and the \directproduct.

\begin{lemma}\label{Lemma_AgeProducts}
	Let $L_0$ and $L_1$ be relational languages, let $A$ be an $L_0$-structure, and let $B$ be an $L_1$-structure.  Then,
	\begin{enumerate}
		\item $\age(A \wreath B) = \age(A) \wreath \age(B)$; and
		\item $\age(A \btimes B) = \age(A) \btimes \age(B)$.
	\end{enumerate}
\end{lemma}

\begin{proof}
	\if\papermode0
		(1): Fix $C \in \age(A \wreath B)$.  Then, up to isomorphism, $C \subseteq A \wreath B$ and $C$ is finite.  Clearly there exists $D \subseteq B$ finite and, for each $d \in D$, $E_d \subseteq A$ finite and non-empty such that $C = \bigsqcup_{d \in D} E_d$.  Thus, $C \in \age(A) \wreath \age(B)$.  Conversely, if $C \in \age(A) \wreath \age(B)$, then there exists $D \in \age(B)$ and $E_d \in \age(A)$ for each $d \in D$ such that $C = \bigsqcup_{d \in D} E_d$.  Up to isomorphism, we may assume $D \subseteq B$ and $E_d \subseteq A$ for each $d \in D$.  Therefore, $C \in \age(A \wreath B)$.
	
		(2): Fix $C \in \age(A \btimes B)$.  Then, up to isomorphism, $C \subseteq A \btimes B$ and $C$ is finite.  Clearly there exists $E \subseteq A$ and $D \subseteq B$ finite such that $C \subseteq E \btimes D$.  Thus, $C \in \age(A) \btimes \age(B)$.  Conversely, if $C \in \age(A) \btimes \age(B)$, then there exists $E \in \age(A)$ and $D \in \age(B)$ such that $C \subseteq E \btimes D$.  Up to isomorphism, we may assume $E \subseteq A$ and $D \subseteq B$.  Therefore, $C \in \age(A \btimes B)$.
	\else
		Follows from definition.
	\fi
\end{proof}

Using Proposition \ref{Proposition_Homogeneity} and Lemma \ref{Lemma_AgeProducts}, together with \fraisse's Theorem, we obtain the following result.

\begin{proposition}\label{Proposition_Fraisse}
	For each $t < 2$, let $L_t$ be a relational language and let $\KK_t$ be a \fraisse~class of finite $L_t$-structures closed under isomorphism.  Then,
	\begin{enumerate}
		\item $\KK_0 \wreath \KK_1$ is a \fraisse~class and
			\[
				\Flim(\KK_0 \wreath \KK_1) = \Flim(\KK_0) \wreath \Flim(\KK_1).
			\]
		\item $\KK_0 \btimes \KK_1$ is a \fraisse~class and
			\[
				\Flim(\KK_0 \btimes \KK_1) = \Flim(\KK_0) \btimes \Flim(\KK_1).
			\]
	\end{enumerate}
\end{proposition}

\begin{proof}
	(1): Since $\Flim(\KK_0)$ and $\Flim(\KK_1)$ are homogeneous, by Proposition \ref{Proposition_Homogeneity}, $\Flim(\KK_0) \wreath \Flim(\KK_1)$ is homogeneous.  It is clearly countable.  Finally, by Lemma \ref{Lemma_AgeProducts},
	\[
		\age(\Flim(\KK_0) \wreath \Flim(\KK_1)) = \age(\Flim(\KK_0)) \wreath \age(\Flim(\KK_1)) = \KK_0 \wreath \KK_1.
	\]
	Therefore, by the uniqueness of \fraisse~limits,
	\[
		\Flim(\KK_0 \wreath \KK_1) = \Flim(\KK_0) \wreath \Flim(\KK_1).
	\]
	
	(2): Follows similarly as (1).
\end{proof}

\begin{example}\label{Example_FreeFraisse}
	The analogue of Proposition \ref{Proposition_Fraisse} does not hold for free superposition.  Let $L$ be the language consisting of a single unary predicate, $P$, and let $\KK$ be the class of all finite $L$-structures $A$ where $|P(A)| \le 1$.  Then, it is easy to check that $\KK$ is a \fraisse~class.  However, $\KK * \KK$ does not have the joint embedding property (nor the amalgamation property).  For more information, see Example 3.4 of \cite{GuiPar}.
\end{example}

If we assume further that the classes have the strong amalgamation property, then the free superposition is a \fraisse~class.

\begin{proposition}[(3.9) of \cite{Cam90}]\label{Proposition_Known}
	If $\KK_0$ and $\KK_1$ are \fraisse~classes with the strong amalgamation property, then $\KK_0 * \KK_1$ is a \fraisse~class (with the strong amalgamation property).
\end{proposition}

\subsection{\fraisse~Properties and Products of Classes} 

By Proposition \ref{Proposition_Fraisse}, the conjunction of the hereditary property, the joint embedding property, and the amalgamation property is preserved under \wreathproduct s and \directproduct s.  This is noted for \wreathproduct s in Lemma 3.3 of \cite{sokQuotients}.  One might ask whether each of these properties, in isolation, is preserved.  It is easy to see that the hereditary property is preserved under \wreathproduct s.  By definition, regardless of whether the component classes have the hereditary property, the full product will.  Next, we will see that the amalgamation property is also preserved under \wreathproduct s and \directproduct s, under mild assumptions.  The joint embedding property works similarly.

\begin{proposition}\label{Proposition_WreathDirectAP}
	For each $t < 2$, let $L_t$ be a relational language and let $\KK_t$ be a class of finite $L_t$-structures closed under isomorphism.
	\begin{enumerate}
		\item If $\KK_0$ has the joint embedding property and $\KK_0$ and $\KK_1$ have the amalgamation property, then $\KK_0 \wreath \KK_1$ has the amalgamation property.
		\item If $\KK_0$ and $\KK_1$ have the hereditary property and the amalgamation property, then $\KK_0 \btimes \KK_1$ has the amalgamation property.
	\end{enumerate}
\end{proposition}

First, we exhibit a lemma about decomposing embeddings between lexicographic products.  This is a generalization of Lemma 2.13 of \cite{Meir22}.

\begin{lemma}\label{Lemma_WreathEmbeddings}
	Fix relational languages $L_0$ and $L_1$ and let $L_2$ be the language of the lexicographic product.  Let $B$ and $D$ be $L_1$-structures.  For each $b \in B$, let $A_b$ be a non-empty $L_0$-structure, and, for each $d \in D$, let $C_d$ be a non-empty $L_0$-structure.  For all $f \in \emb_{L_2}\left( \bigsqcup_{b \in B} A_b , \bigsqcup_{d \in D} C_d \right)$, there exists $g \in \emb_{L_1}(B, D)$ and, for each $b \in B$, $h_{g(b)} \in \emb_{L_0}(A_b, C_{g(b)})$ such that, for all $b \in B$ and $a \in A_b$,
	\[
		f(a,b) = (h_{g(b)}(a), g(b)).
	\]
	Moreover, for any such $g$ and $(h_{g(b)} : b \in B)$, constructing $f$ in this manner gives an $L_2$-embedding from $\bigsqcup_{b \in B} A_b$ to $\bigsqcup_{d \in D} C_d$.
\end{lemma}
	
\begin{proof}[Proof of Proposition~\ref{Proposition_WreathDirectAP}]
	(1): Fix $A, B_0, B_1 \in \KK_0 \wreath \KK_1$ and embeddings $f_0 \in \emb(A, B_0)$ and $f_1 \in \emb(A, B_1)$.  Then, there exists $A^1, B^1_0, B^1_1 \in \KK_1$, $A^0_a \in \KK_0$ for each $a \in A^1$, $B^0_{t,b} \in \KK_0$ for each $t < 2$ and $b \in B^1_t$ such that $A = \bigsqcup_{a \in A^1} A^0_a$ and $B_t = \bigsqcup_{b \in B^1_t} B^0_{t,b}$ for each $t < 2$.  Moreover, by Lemma \ref{Lemma_WreathEmbeddings}, there exists $f^1_t \in \emb(A^1, B^1_t)$ for each $t < 2$ and $f^0_{t,a} \in \emb(A^0_a, B^0_{t,f^1_t(a)})$ for each $t < 2$ and $a \in A^1$ such that, for all $a \in A^1$ and $a^* \in A^0_a$, $f_t(a^*,a) = (f^0_{t,a}(a^*), f^1_t(a))$.  Since $\KK_1$ has the amalgamation property, there exists $C^1 \in \KK_1$ and $g^1_t \in \emb(B^1_t, C^1)$ for each $t < 2$ such that $g^1_0 \circ f^1_0 = g^1_1 \circ f^1_1$.  For each $c \in C^1$, we have four cases:
	\begin{itemize}
		\item If $c \in C^1 \setminus \left( \bigcup_{t < 2} g^1_t(B^1_t) \right)$, then set $C^0_c$ to be any element of $\KK_0$.
		\item If, for some $t < 2$, $c \in g^1_t(B^1_t) \setminus g^1_{1-t}(B^1_{1-t})$, then set $C^0_c = B^0_b$ for any choice of $b \in B^1_t$ such that $g^1_t(b) = c$ and let $g^0_{t,b}$ be the identity embedding on $B^0_b$.
		\item If $c \in (g^1_0(B^1_0) \cap g^1_1(B^1_1)) \setminus g^1_0(f^1_0(A^1))$, then let $C^0_c \in \KK_0$ and let $g^0_{t,b_t} \in \emb(B^0_{t,b_1}, C^0_c)$ for each $t < 2$, where $b_t \in B^1_t$ is such that $c = g^1_t(b_t)$.  This exists since $\KK_0$ has the joint embedding property.
		\item If $c \in g^1_0(f^1_0(A^1))$, then let $C^0_c \in \KK_0$ and $g^0_{t,f^1_t(a)} \in \emb(B^0_{t,f^1_t(a)})$ for each $t < 2$ such that
		\[
			g^0_{0,f^1_0(a)} \circ f^0_{0,a} = g^0_{1,f^1_1(a)} \circ f^0_{1,a},
		\]
		where $a \in A^1$ is such that $c = g^1_0(f^1_0(a))$.  This exists since $\KK_0$ has the amalgamation property.
	\end{itemize}
	Set $C = \bigsqcup_{c \in C^1} C^0_c$ and, for each $t < 2$, define $g_t : B_t \rightarrow C$ by setting, for all $b \in B^1_t$ and $b^* \in B^0_{t,b}$,
	\[
		g_t(b^*, b) = (g^0_{t,b}(b^*), g^1_t(b)).
	\]
	One can see that $g_t \in \emb(B_t, C)$ for each $t < 2$ and that $g_0 \circ f_0 = g_1 \circ f_1$.
	
	(2): Fix $A, B_0, B_1 \in \KK_0 \btimes \KK_1$ and embeddings $f_0 \in \emb(A, B_0)$ and $f_1 \in \emb(A, B_1)$.  Then, there exists $A^0, B^0_0, B^0_1 \in \KK_0$ and $A^1, B^1_0, B^1_1 \in \KK_1$ such that $A \subseteq A^0 \btimes A^1$ and $B_t \subseteq B^0_t \btimes B^1_t$ for each $t < 2$.  Moreover, choose these structures minimal such.  By Lemma \ref{Lemma_DirectEmbeddings}, there exists $f^s_t \in \emb(A^s, B^s_t)$ for each $t, s < 2$ such that, for all $a_0 \in A^0$ and $a_1 \in A^1$, $f_t(a_0, a_1) = (f^0_t(a_0), f^1_t(a_1))$ for each $t < 2$.  For each $s < 2$, since $\KK_s$ has the amalgamation property, there exists $C^s \in \KK_s$ and $g^s_t \in \emb(B^s_t, C^s)$ for each $t < 2$ such that $g^s_0 \circ f^s_0 = g^s_1 \circ f^s_1$.  Let $C = C^0 \btimes C^1$ and, for each $t < 2$, define $g_t : B_t \rightarrow C$ by setting, for $b_0 \in B^0_t$ and $b_1 \in B^1_t$,
	\[
		g_t(b_0, b_1) = (g^0_t(b_0), g^1_t(b_1)).
	\]
	One can see that $g_t \in \emb(B_t, C)$ for each $t < 2$ and that $g_0 \circ f_0 = g_1 \circ f_1$.
\end{proof}

\subsection{\AgeIndivisibility\ and Products of Classes}

\Ageindivisibility\ is preserved under the \wreathproduct\ and under the \directproduct\ (with some additional assumptions).

\begin{theorem}\label{Theorem_WreathAgeInd}
	For each $t < 2$, let $L_t$ be a relational language and let $\KK_t$ be a class of finite $L_t$-structures closed under isomorphism.  If $\KK_0$ has the joint embedding property, $\KK_0$ is \ageindivisible, and $\KK_1$ is \ageindivisible, then $\KK_0 \wreath \KK_1$ is \ageindivisible.
\end{theorem}

\begin{proof}
    Fix $A \in \KK_0 \wreath \KK_1$ and $k < \omega$.  Therefore, there exists $B \in \KK_1$ and $D_b \in \KK_0$ for all $b \in B$ such that $A = \bigsqcup_{b \in B} D_b$.  Since $\KK_0$ has the joint embedding property and $B$ is finite, there exists $D \in \KK_0$ and $f_b \in \emb_{L_0}(D_b, D)$ for all $b \in B$.  Since $\KK_0$ is \ageindivisible, there exists $D' \in \KK_0$ such that, for all $c : D' \rightarrow k$, there exists an $g \in \emb_{L_0}(D, D')$ such that $|c(g(D))| = 1$.  Since $\KK_1$ is \ageindivisible, there exists $B' \in \KK_1$ such that, for all $c : B' \rightarrow k$, there exists $h \in \emb_{L_1}(B, B')$ such that $|c(h(B))| = 1$.  Let $A' = D' \wreath B'$.  We claim that this works.
    
    Fix $c : A' \rightarrow k$.  For each $b \in B'$, let $c_b : D' \rightarrow k$ be given by, for all $d \in D'$, $c_b(d) = c(d,b)$.  Therefore, there exists $g_b \in \emb_{L_0}(D, D')$ such that $|c_b(g_b(D))| = 1$.  Let $t_b < k$ be such that $c_b(g_b(D)) = \{ t_b \}$.  Define $c' : B' \rightarrow k$ by setting, for each $b \in B'$, $c'(b) = t_b$.  By construction, there exists $h \in \emb_{L_1}(B, B')$ such that $|c'(h(B))| = 1$.  Let $t < k$ be such that $c'(h(B)) = \{ t \}$.  Finally, define $j : A \rightarrow A'$ by setting, for each $b \in B$ and $d \in D_b$,
    \[
        j(d,b) = (g_{h(b)}(f_b(d)), h(b)).
    \]
	We claim that $j \in \emb(A,A')$ and that $|c(j(A))| = 1$.
        
	Fix $b \in B$ and $d \in D_b$.  Then,
	\begin{align*}
		c(j(d,b)) = & \ c(g_{h(b)}(f_b(d)), h(b)) = \\ & \ c_{h(b)}(g_{h(b)}(f_b(d))) = t_{h(b)} = c'(h(b)) = t.
	\end{align*}
	Therefore, $c(j(A)) = \{ t \}$.
	\if\papermode0
		Next, fix $R \in \sig(L_0)$ of arity $n$, $\ob \in B^n$, and $\od$ such that $d_i \in D_{b_i}$ for all $i < n$.  Then,
        \begin{align*}
            A \models R(\od, \ob) & \Longleftrightarrow b_i = b_j \text{ for all $i, j$ and } D_{b_0} \models R(\od) \\ & \Longleftrightarrow h(b_i) = h(b_j) \text{ for all $i, j$ and } D' \models R(g_{h(b_0)}(f_{b_0}(\od))) \\ & \Longleftrightarrow A' \models R(j(\od, \ob)).
        \end{align*}
        Fix $R \in \sig(L_1)$ of arity $n$, $\ob \in B^n$, and $\od$ such that $d_i \in D_{b_i}$ for all $i < n$.  Then,
        \begin{align*}
            A \models R(\od, \ob) & \Longleftrightarrow B \models R(\ob) \\ & \Longleftrightarrow B' \models R(h(\ob)) \Longleftrightarrow A' \models R(j(\od, \ob)).
        \end{align*}
        Finally, for all $b_0, b_1 \in B$, $d_0 \in D_{b_0}$, and $d_1 \in D_{b_1}$,
        \begin{align*}
            A \models E(d_0, b_0, d_1, b_1) & \Longleftrightarrow b_0 = b_1 \Longleftrightarrow h(b_0) = h(b_1) \\ & \Longleftrightarrow D \models E(j(d_0, b_0), j(d_1, b_1)).
        \end{align*}
        Therefore, $j$ is an embedding.  This completes the proof.
    \else
        It is easy to see that $j \in \emb(A,A')$.
    \fi
\end{proof}

\begin{example}
	In Theorem \ref{Theorem_WreathAgeInd}, the assumption of $\KK_0$ having the joint embedding property is necessary.  For example, let $L_0$ be the language with two binary relation symbols, $R_0$ and $R_1$, and let $L_1$ be the language with empty signature.  Let $\KK_0$ be the class of all finite $L_0$-structures $A$ where $R_0$ and $R_1$ are symmetric and irreflexive on $A$ and
	\[
		A \models \neg (\exists x, y, z, w)( R_0(x,y) \wedge R_1(z,w)).
	\]
	In other words, $\KK_0$ is the union of the class of all finite $R_0$-graphs and the class of all finite $R_1$-graphs.  Let $\KK_1$ be the class of all finite sets.  Then, it is clear that $\KK_0$ does not have the joint embedding property.  Moreover, $\KK_0$ is \ageindivisible\ (because $\GG$ is \ageindivisible) and $\KK_1$ is \ageindivisible.  However, $\KK_0 \wreath \KK_1$ is not \ageindivisible.
	
	Take $A \in \KK_0 \wreath \KK_1$ where $A$ is composed of two $E$-classes.  The first $E$-class consists of an $R_0$-path of length $1$ and the other $E$-class consists of an $R_1$-path of length $1$.  Then, it is easy to check that, for any $B \in \KK_0 \wreath \KK_1$, if we color $c : B \rightarrow 2$ by setting $c(b) = 0$ if and only if $b$ belongs to an $E$-class with an $R_0$-edge, then there exists no $f \in \emb(A,B)$ such that $|c(f(A))| = 1$.
\end{example}

\begin{theorem}\label{Theorem_DirectAgeInd}
    For each $t < 2$, let $L_t$ be a relational language and let $\KK_t$ be a class of finite $L_t$-structures closed under isomorphism.  If $\KK_0$ and $\KK_1$ are \ageindivisible, then so is $\KK_0 \btimes \KK_1$.
\end{theorem}

\begin{proof}
    Fix $A \in \KK_0 \btimes \KK_1$ and $k < \omega$.  So there exists $D \in \KK_0$ and $B \in \KK_1$ such that $A \subseteq D \btimes B$.  Since $\KK_0$ is \ageindivisible, there exists $D' \in \KK_0$ such that, for all $c : D' \rightarrow k$, there exists $f \in \emb_{L_0}(D,D')$ such that $|c(f(D))| = 1$.  Since $\KK_1$ is \ageindivisible, there exists $B'$ such that, for all $c : B' \rightarrow k \times \emb_{L_0}(D, D')$, there exists $f \in \emb_{L_1}(B, B')$ such that $|c(f(B))| = 1$.  Let $A' = D' \btimes B'$.  Clearly $A' \in \KK_0 \btimes \KK_1$.  We claim that this works.
    
    Let $c : A' \rightarrow k$.  For each $b \in B'$, consider the coloring $c_b : D' \rightarrow k$ given by, for all $d \in D'$, $c_b(d) = c(d,b)$.  By assumption, there exists $f_b \in \emb_{L_0}(D, D')$ such that $|c_b(f_b(D))| = 1$.  Let $t_b < k$ be such that $c_b(f_b(D)) = \{ t_b \}$.  Next, define $c' : B' \rightarrow k \times \emb_{L_0}(D, D')$ by setting, for all $b \in B'$, $c'(b) = (t_b, f_b)$.  By assumption, there exists $g \in \emb_{L_1}(B, B')$ such that $|c'(g(B))| = 1$.  That is, there exists $f \in \emb_{L_0}(D,D')$ and $t < k$ such that, for all $b \in B$, $c'(g(b)) = (t, f)$.  Let $h : A \rightarrow A'$ be given by, for all $(d,b) \in A$,
    \[
    	h(d,b) = (f(d), g(b)).
    \]
    Clearly $h \in \emb(A, A')$.  Moreover, for each $(d,b) \in A$, $c'(g(b)) = (t, f)$, thus $t_{g(b)} = t$ and $f_{g(b)} = f$.  Therefore,
    \begin{align*}
        c(h(d,b)) = & \ c(f(d), g(b)) = c_{g(b)}(f(d)) = \\ & \ c_{g(b)}(f_{g(b)}(d)) = t_{g(b)} = t.
    \end{align*}
    Therefore, $c(h(A)) = \{ t \}$.
\end{proof}

Together Theorem \ref{Theorem_WreathAgeInd} and Theorem \ref{Theorem_DirectAgeInd} say that, at least when the classes under consideration have the joint embedding property (e.g., are the ages of some structure), then the \wreathproduct\ and the \directproduct\ preserve \ageindivisibility.  Whether or not this holds true for the free superposition is an open question.

\begin{question}\label{Question_FreeAgeInd}
	For each $t < 2$, let $L_t$ be a relational language and let $\KK_t$ be a class of finite $L_t$-structures.  Under what conditions on $\KK_0$ and $\KK_1$ do we have that, if $\KK_0$ and $\KK_1$ are \ageindivisible, then so is $\KK_0 * \KK_1$?
\end{question}

Free superposition does not preserve the indivisibility of \fraisse~limits.

\begin{example}\label{Example_FreeIndivis}
	Consider $\EE$, the class of all finite sets with a single equivalence relation.  Even though $\Flim(\EE)$ is indivisible (and $\EE * \EE$ is a \fraisse~class), $\Flim(\EE * \EE)$ is not indivisible.  See, for instance, Example 3.19 in \cite{GuiPar}.
\end{example}

\subsection{Definable Self-Similarity and Products of Classes}\label{Subsection_DSS}

Neither \wreathproduct s nor \directproduct s preserve definable self-similarity.

\begin{example}\label{Example_NotDefSelfSim}
	Let $\KK$ be the class of all finite sets with no structure.  It is easy to verify that $\KK$ is definably self-similar.  However, both $\KK \wreath \KK$ and $\KK \btimes \KK$ are not definably self-similar.
	
	For $\KK \wreath \KK$, let $A$ be a singleton, let $B$ be contain two $E$-inequivalent elements, and let $C$ contain two $E$-equivalent elements.  Take $C_0$ to be a singleton and $c$ to be the element of $C$ not in $C_0$.  Let $f$ be any embedding of $A$ into $B$ and let $g$ be an embedding of $A$ into $\{ c \}$.  Then, it is easy to check that there is no $D$, $j \in \emb(C,D)$, and $h \in \emb(B, \qfclass{j(c)}{D}{j(C_0)})$ such that $h \circ f = j \circ g$.
	
	For $\KK \btimes \KK$, let $A$ be a singleton, let $B = \{ b_0, b_1 \}$ with $B \models E_0(b_0, b_1) \wedge \neg E_1(b_0, b_1)$, and let $C = \{ c, c_0 \}$ with $C \models \neg E_0(c, c_0) \wedge E_1(c, c_0)$.  Let $f$ be the embedding of $A$ into $\{ b_0 \}$ and $g$ be the embedding of $A$ into $\{ c \}$.  Then, it is easy to check that there is no $D$, $j \in \emb(C,D)$, and $h \in \emb(B,\qfclass{j(c)}{D}{\{j(c_0)\}})$ such that $h \circ f = j \circ g$.
\end{example}

As long as $\KK$ has the hereditary property, free superposition preserves definable self-similarity.

\begin{proposition}[Proposition 3.16 of \cite{GuiPar}]\label{Proposition_FreeDefSelfSim}
	For each $t < 2$, let $L_t$ be a relational language and let $\KK_t$ be a class of finite $L_t$-structures with the hereditary property.  If $\KK_0$ and $\KK_1$ are definably self-similar, then so is $\KK_0 * \KK_1$.
\end{proposition}

\if\papermode0
	\begin{proof}
		Let $A, B, C \in \KK_0 * \KK_1$, $f \in \emb_{L_2}(A,B)$, $p(x)$ a complete, quantifier-free, non-algebraic $1$-$L_2$-type over $C_0 \subseteq C$, and $g \in \emb_{L_2}(A,p(C))$.  Since $\KK_0 * \KK_1$ has the hereditary property, we may assume $C = g(A) \cup C_0$.  In particular, for each $t < 2$, $f$ is an $L_t$-embedding from $A |_{L_t}$ to $B |_{L_t}$, $p_t := p |_{L_t}$ is a complete, quantifier-free, non-algebraic $1$-$L_t$-type over $C_0$, and $g$ is an $L_t$-embedding from $A |_{L_t}$ to $p_t(C |_{L_t})$.  Since $\KK_t$ is definably self-similar, there exists $D_t \in \KK_t$, $j_t \in \emb_{L_t}( C |_{L_t}, D_t )$, and $h_t \in \emb_{L_t}(B |_{L_t} j_t(p_t)(D_t))$ such that $h_t \circ f = j_t \circ g$.  Since $\KK_t$ has the hereditary property, we may assume that $D_t = h_t(B) \cup j_t(C)$.  Thus, there exists a bijection $k : D_0 \rightarrow D_1$ such that $h_1 = k \circ h_0$ and $j_1 = k \circ j_0$.  Let $D$ be the $L_2$-structure on $D_0$ given via $k$.  It is easy to check that $j_0 \in \emb_{L_2}(C,D)$ and $h_0 \in \emb_{L_2}(B,j(p)(D))$.
	\end{proof}
\fi

Although Proposition 3.16 of \cite{GuiPar} assumes that $L_0$ and $L_1$ are finite and $\KK_0$ and $\KK_1$ are \fraisse~classes with the strong amalgamation property, these assumptions are clearly not used in the proof provided.  Moreover, the proof uses the definition of definable self-similarity given in this paper (i.e., the condition from Lemma 3.14 of \cite{GuiPar}), so no limit structure is needed.

\subsection{Disjoint $n$-amalgamation Property and Products of Classes}

The disjoint $n$-amalgamation property is preserved under free superposition for all $n$, is preserved under the \wreathproduct\ for only $n = 2$, and is preserved under the \directproduct\ for no $n$.

\begin{example}\label{Example_WreathDirectnAmalg}
	For $n \ge 3$, the disjoint $n$-amalgamation property is not preserved under \wreathproduct s or \directproduct s.  Indeed, if $\KK$ is the class of all finite sets with no structure, then clearly $\KK$ has the disjoint $n$-amalgamation property for all $n$.  However, $\KK \wreath \KK$ and $\KK \btimes \KK$ satisfy the conditions of Lemma \ref{Lemma_TransitiveNo3Amalg}.  Hence, they do not have the disjoint $n$-amalgamation property for any $n \ge 3$.
\end{example}

\begin{example}\label{Example_Direct2Amalg}
	The disjoint $2$-amalgamation property is not preserved under \directproduct s.  Let $\KK$ be the class of all finite sets with no structure and consider $\KK \btimes \KK$.  Let $A = \{ (0,0), (1,1) \}$ and $B_0 = B_1 = \{ (0,0), (1,1), (0,1) \}$, each equipped with the usual structure, and let $f_0 \in \emb(A, B_0)$ and $f_1 \in \emb(A, B_1)$ be the identity embeddings.  Then, there exists no $C \in \KK \btimes \KK$ and embeddings $g_0 \in \emb(B_0, C)$ and $g_1 \in \emb(B_1, C)$ such that $g_0 \circ f_0 = g_1 \circ f_1$ and $g_0(B_0) \cap g_1(B_1) = g_0(f_0(A))$ (there exists at most one element $c \in C$ such that $C \models E_0(c, g_0(f_0(0,0))) \wedge E_1(c, g_0(f_0(1,1)))$).
\end{example}

Though the disjoint $n$-amalgamation property is not preserved under \wreathproduct s for $n \ge 3$, it \emph{is} preserved when $n = 2$.  In other words, the strong amalgamation property is preserved under \wreathproduct s.

\begin{theorem}\label{Theorem_WreathStrongAmalg}
	For each $t < 2$, let $L_t$ be a relational language and let $\KK_t$ be a class of finite $L_t$-structures closed under isomorphism.  If $\KK_0$ and $\KK_1$ have the strong amalgamation property, then $\KK_0 \wreath \KK_1$ has the strong amalgamation property.
\end{theorem}

\begin{proof}
    \if\papermode0
    	Fix $A, B_0, B_1 \in \KK_0 \wreath \KK_1$ and embeddings $f_0 \in \emb(A, B_0)$ and $f_1 \in \emb(A, B_1)$.  Then, there exists $A^1, B^1_0, B^1_1 \in \KK_1$, $A^0_a \in \KK_0$ for each $a \in A^1$, $B^0_{t,b} \in \KK_0$ for each $t < 2$ and $b \in B^1_t$ such that $A = \bigsqcup_{a \in A^1} A^0_a$ and $B_t = \bigsqcup_{b \in B^1_t} B^0_{t,b}$ for each $t < 2$.  Moreover, by Lemma \ref{Lemma_WreathEmbeddings}, there exists $f^1_t \in \emb(A^1, B^1_t)$ for each $t < 2$ and $f^0_{t,a} \in \emb(A^0_a, B^0_{t,f^1_t(a)})$ for each $t < 2$ and $a \in A^1$ such that, for all $a \in A^1$ and $a^* \in A^0_a$, $f_t(a^*,a) = (f^0_{t,a}(a^*), f^1_t(a))$.  Since $\KK_1$ has the  strong amalgamation property, there exists $C^1 \in \KK_1$ and $g^1_t \in \emb(B^1_t, C^1)$ for each $t < 2$ such that
    	\[
	   	g^1_0 \circ f^1_0 = g^1_1 \circ f^1_1 \text{ and } g^1_0(B^1_0) \cap g^1_1(B^1_1) = g^1_0(f^1_0(A^1)).
    	\]
    	For each $c \in C^1$, we have three cases:
    	\begin{itemize}
    		\item If $c \in C^1 \setminus \left( \bigcup_{t < 2} g^1_t(B^1_t) \right)$, then set $C^0_c$ to be any element of $\KK_0$.
    		\item If, for some $t < 2$, $c \in g^1_t(B^1_t) \setminus g^1_{1-t}(B^1_{1-t})$, then set $C^0_c = B^0_b$ where $b \in B^1_t$ is such that $g^1_t(b) = c$ and let $g^0_{t,b}$ be the identity embedding on $B^0_b$.
    		\item If $c \in g^1_0(f^1_0(A^1))$, then let $C^0_c \in \KK_0$ and $g^0_{t,f^1_t(a)} \in \emb(B^0_{t,f^1_t(a)}, C^0_c)$ for each $t < 2$ such that
    			\begin{align*}
    				& g^0_{0,f^1_0(a)} \circ f^0_{0,a} = g^0_{1,f^1_1(a)} \circ f^0_{1,a}, \text{ and } \\ & g^0_{0,f^1_0(a)}(B^0_{0,f^1_0(a)}) \cap g^0_{1,f^1_1(a)}(B^0_{1,f^1_1(a)}) = g^0_{0,f^1_0(a)}(f^0_{0,a}(A^0_a)),
    			\end{align*}
    			where $a \in A^1$ is such that $c = g^1_0(f^1_0(a))$.  This exists since $\KK_0$ has the  strong amalgamation property.
	   \end{itemize}
    	Set $C = \bigsqcup_{c \in C^1} C^0_c$ and, for each $t < 2$, define $g_t : B_t \rightarrow C$ by setting, for all $b \in B^1_t$ and $b^* \in B^0_{t,b}$,
	   \[
    		g_t(b^*, b) = (g^0_{t,b}(b^*), g^1_t(b)).
	   \]
    	One can see that $g_t \in \emb(B_t, C)$ for each $t < 2$, that $g_0 \circ f_0 = g_1 \circ f_1$, and that $g_0(B_0) \cap g_1(B_1) = g_0(f_0(A))$.
    \else
        This is similar to the proof of Proposition \ref{Proposition_WreathDirectAP} (1), except we add the assumption that $g^1_0(B^1_0) \cap g^1_1(B^1_1) = g^1_0(f^1_0(A^1))$ (in the notation of the proof).  From this, we can infer that $g_0(B_0) \cap g_1(B_1) = g_0(f_0(A))$.
    \fi
\end{proof}

For all $n \ge 2$, the disjoint $n$-amalgamation property is preserved under free superposition (so long as the classes in question have the hereditary property).

\begin{theorem}\label{Theorem_FreenAmalgamation}
	For each $t < 2$, let $L_t$ be a relational language and let $\KK_t$ be a class of finite $L_t$-structures closed under isomorphism with the hereditary property.  For all $n \ge 2$, if $\KK_0$ and $\KK_1$ have the disjoint $n$-amalgamation property, then $\KK_0 * \KK_1$ has the disjoint $n$-amalgamation property.
\end{theorem}

\begin{proof}
	Suppose that $A_X \in \KK_0 * \KK_1$ for $X \subset n$ and $f_{X,Y} \in \emb(A_X, A_Y)$ for $X \subseteq Y \subset n$ is a disjoint $(\Pow(n) \setminus \{ n \})$-amalgamation system in $\KK_0 * \KK_1$.  Then, for each $t < 2$, $((A_X |_{L_t})_{X \subset n}, (f_{X,Y})_{X \subseteq Y \subset n})$ is a disjoint $(\Pow(n) \setminus \{ n \})$-amalgamation system in $\KK_t$.  Since $\KK_t$ has disjoint $n$-amalgamation, there exists an extension of this to a disjoint $\Pow(n)$-amalgamation system in $\KK_t$, say with $A_{n,t}$ and $f_{X,n,t} \in \emb(A_X |_{L_t}, A_{n,t})$ for $X \subset n$ completing the system.  Since $\KK_t$ has the hereditary property, we may assume that $A_{n,t} = \bigcup_{X \subset n} f_{X,n,t}(A_X)$.  Define $g : A_{n,0} \rightarrow A_{n,1}$ as follows:
	
	Fix $a \in A_{n,0}$.  Fix $X \subset n$ minimal such that $a \in f_{X,n,0}(A_X)$.  Such an $X$ is unique; if there were $X \neq X'$ so that $a \in f_{X,n,0}(A_X) \cap f_{X,n,0}(A_{X'})$, then, by disjointness, $a \in f_{X \cap X',n,0}(A_{X \cap X'})$.  Let $b \in A_X$ be such that $a = f_{X,n,0}(b)$ and set $g(a) = f_{X,n,1}(b)$.  Note that $X$ is also minimal such that $g(a) \in f_{X,n,1}(A_X)$; if $g(a) \in f_{X',n,1}(A_{X'})$ with $X' \subseteq X$, then, by commutivity, $a \in f_{X',n,0}(A_{X'})$, hence $X' = X$ by minimality of $X$.  Therefore, this process is reversible, and $g$ is a bijection.
	
	We can endow $A_{n,0}$ with an $L_1$-structure via $g$; call it $A_n$.  By construction, for all $X \subset n$, $g \circ f_{X,n,0} = f_{X,n,1}$.  Thus, $f_{X,n,0} \in \emb(A_X, A_n)$ and this is a disjoint $\Pow(n)$-amalgamation system in $\KK_0 * \KK_1$.
\end{proof}

\section{Configurations}\label{Section_Configurations}

In this section we refer back to the model-theoretic themes in the Introduction, namely, the new framework for studying model-theoretic dividing lines called \emph{configurations} developed in \cite{GuiHill} and \cite{GuiPar}.  This framework led to a natural partial ordering on dividing lines; see Remark \ref{Remark_DividingLines}.  Some open questions remained in terms of our understanding of this partial order.

Since configurations study the relationship of a combinatorial object to the model ``containing'' it, we will employ two languages, an \emph{index} language, $L_0$, and a \emph{target} language, $L$.  We demand that $L_0$ is a relational language, but we allow $L$ to be any language.  Our index object $\KK$  will be a class of finite $L_0$-structures closed under isomorphism.  Our target object $M$ will be  an $L$-structure.  We study configurations from the index $\KK$ to the target $M$.

\begin{definition}[Configuration]\label{Definition_Configuration}
    Let $L_0$ be a relational language and let $\KK$ be a class of finite $L_0$-structures closed under isomorphism.  Let $L$ be a language, $M$ an $L$-structure, and $n < \omega$.  A \emph{$\KK$-configuration into $M^n$} is a sequence of functions $(I, f_A)_{A \in \KK}$ such that
    \begin{enumerate}
        \item $I : \sig(L_0) \rightarrow \Fml(L,M)$,
        \item for all $A \in \KK$, $f_A : A \rightarrow M^n$, and
        \item for all $R \in \sig(L_0)$, for all $A \in \KK$, for all $\oa \in A^{\arity(R)}$,
        \[
            A \models R(\oa) \Longleftrightarrow M \models I(R)(f_A(\oa)).
        \]
    \end{enumerate}
    We say that a $\KK$-configuration $(I, f_A)_{A \in \KK}$ is \emph{injective} if $f_A$ is injective for each $A \in \KK$.
    
    If $T$ is an $L$-theory, we say that $T$ \emph{admits a $\KK$-configuration} if there exists $M \models T$, $n < \omega$, and a $\KK$-configuration into $M^n$.  Let $\CC_\KK$ denote the class of all complete theories with infinite models that admit a $\KK$-configuration.
\end{definition}

\begin{remark}\label{Remark_OldConfigurations}
	Definition \ref{Definition_Configuration} is a generalization of the notion of $\KK$-configuration from \cite{GuiPar}.  Suppose that $L_0$ is a relational language and $\Gamma$ is an $L_0$-structure.  Suppose that $L$ is a language, $M$ is a $|\Gamma|^+$-saturated $L$-structure, and $n < \omega$.  Then, by the proof of Lemma 4.2 of \cite{GuiPar}, there exists an $\age(\Gamma)$-configuration into $M^n$ if and only if there exists $I : \sig(L_0) \rightarrow \Fml(L,M)$ and $f : \Gamma \rightarrow M^n$ such that, for all $R \in \sig(L_0)$ and $\oa \in \Gamma^{\arity(R)}$,
    \[
        \Gamma \models R(\oa) \Longleftrightarrow M \models I(R)(f(\oa)).
    \]   
\end{remark}

\if\papermode0
	\begin{proof}
		Consider the $L$-type $\Sigma$ in the variables $(\ox_a)_{a \in \Gamma}$ (where $|\ox_a| = n$) given by
		\[
			\Sigma = T \cup \{ I(R)(\ox_\oa)^{\mathiff \Gamma \models R(\oa)} : R \in \sig(L_0), \oa \in \Gamma^{\arity(R)} \}.
		\]
		If there exists an $\age(\Gamma)$-configuration into $M^n$, then, by compactness, $\Sigma$ is consistent.  Since $M$ is $|\Gamma|^+$-saturated, $\Sigma$ has a realization in $M$.
	\end{proof}
\fi

\begin{remark}[Proposition 4.14 of \cite{GuiPar}]\label{Remark_AgeIndConfig}
	If a class is \ageindivisible, then configurations can be made uniform.  Let $L_0$ be a relational language and $\KK$ be a class of finite $L_0$-structures that has the joint embedding property and is \ageindivisible.  Let $L$ be a language, $M$ be a sufficiently saturated $L$-structure, $C \subseteq M$ small, and $n < \omega$.  If there exists a $\KK$-configuration into $M^n$, then there exists a $\KK$-configuration $(I, f_A)_{A \in \KK}$ into $M^n$ such that, for all $A, B \in \KK$, $a \in A$, and $b \in B$,
    \[
        \tp_L(f_A(a) / C) = \tp_L(f_B(b) / C).
    \]
\end{remark}

The notion of configuration is related to non-collapsing generalized indiscernibles, when $\KK$ is a \fraisse~class with the Ramsey property.

\begin{definition}\label{Definition_GenInd}
    Let $L_0$ be a relational language and let $\Gamma$ be an $L_0$-structure.  Let $L$ be a language and $M$ an $L$-structure.  A \emph{$\Gamma$-generalized indiscernible into $M$} is a function $f : \Gamma \rightarrow M^n$ (for some $n < \omega$) such that, for all $\oa, \ob \in \Gamma^{< \omega}$,
    \[
        \qftp_{L_0}(\oa) = \qftp_{L_0}(\ob) \Longrightarrow \tp_L(f(\oa)) = \tp_L(f(\ob)).
    \]
    We say $f$ is \emph{non-collapsing} if, for all $\oa, \ob \in \Gamma^{< \omega}$,
    \[
        \qftp_{L_0}(\oa) = \qftp_{L_0}(\ob) \Longleftrightarrow \tp_L(f(\oa)) = \tp_L(f(\ob)).
    \]
\end{definition}

\begin{remark}[Theorem 3.14 of \cite{GuiHill}]
    Suppose that $L_0$ is a finite relational language, $\KK$ is a \fraisse~class of $L_0$-structures with the Ramsey property (for embeddings), and $\Gamma$ is the \fraisse~limit of $\KK$.  Then, $\CC_{\KK}$ is the class of all complete theories $T$ with infinite models such that there exists a non-collapsed $\Gamma$-generalized indiscernible into $M$ for some $M \models T$.
\end{remark}

\begin{remark}[Theorem 5.2 of \cite{GuiPar}]\label{Remark_DividingLines}
    We get the following familiar dividing lines:
    \begin{enumerate}
        \item $\CC_{\EE}$ is the class of all complete theories with infinite models.
        \item $\CC_{\LO}$ is the class of all complete, unstable theories.
        \item $\CC_{\GG}$ is the class of all complete theories with the independence property.
        \item For all $k \ge 2$, $\CC_{\HH_k}$ is the class of all complete theories with the $(k-1)$-independence property.
    \end{enumerate}
    In particular, we get the following strict chain of classes:
	\[
		\CC_{\EE} \supset \CC_{\LO} \supset \CC_{\GG} = \CC_{\HH_2} \supset \CC_{\HH_3} \supset \CC_{\HH_4} \supset \dots.
	\]
\end{remark}

This leads to the following question:

\begin{question}\label{Question_Classes}
	Let $\mathbf{P}$ be the class of all $\CC_{\KK}$ as $\KK$ varies over all classes of finite structures, closed under isomorphism, in a relational language.  Then, is $\mathbf{P}$ linearly ordered by inclusion?
\end{question}

In the remainder of this section, we explore some partial results towards answering Question \ref{Question_Classes}.  First, we show that some of the other classes of structures mentioned in Example \ref{Example_BasicExamples} do not create new classes of theories.  Then, we show that considering products of known classes of structures does not create new classes of theories.

\subsection{Other Configuration Classes}

Suppose that $T$ has infinite models.  In the next two lemmas, we show that, if $T$ admits a $\KK$-configuration and $M \models T$ is sufficiently saturated, then there exists an \emph{injective} $\KK$-configuration into $M^n$ for some $n < \omega$.  This will simplify our analysis.

\begin{lemma}\label{Lemma_SaturatedConfig}
    Let $L_0$ be a relational language and let $\KK$ be a class of finite $L_0$-structures closed under isomorphism with fewer than $\kappa$-many isomorphism classes of structures for some cardinal $\kappa \ge \aleph_0$.  Let $L$ be a language and let $T$ be an $L$-theory with infinite models.  If $T$ admits a $\KK$-configuration, $M \models T$, and $M$ is $\kappa$-saturated, then there exists a $\KK$-configuration into $M^n$ for some $n < \omega$.
\end{lemma}

\begin{proof}
	\if\papermode0
	    Since $T$ admits a $\KK$-configuration, there exists $M' \models T$ and $(I, f_A)_{A \in \KK}$ a $\KK$-configuration into $(M')^n$ for some $n$.  Let $K \subseteq \KK$ be a representative set of the isomorphism classes of structures from $\KK$.  By assumption, $|K| < \kappa$ and, for each $A \in K$, $|A| < \aleph_0$.  For each $A \in K$ and $a \in A$, consider a new variable $\oy_{A,a}$ with $|\oy_{A,a}| = n$.  For each $A \in K$, define $\Sigma_A$ an $L$-type over $\emptyset$ in the variables $(\oy_{A,a})_{a \in A}$ as follows:
	    \[
    	    \Sigma_A = \left\{ I(R)( \oy_{a_0}, ... \oy_{a_{\arity(R)-1}} )^{\mathiff A \models R(\oa)} : R \in \sig(L_0), \oa \in A^{\arity(R)} \right\}.
	    \]
    	Finally, let $\Sigma$ be the union of the $\Sigma_A$ for all $A \in K$.  Note that $\Sigma$ is consistent, witnessed by $(f_A(a))_{A \in K, a \in A}$ in $M'$.  Moreover, $\Sigma$ has fewer than $\kappa$-many variables.  Since $M$ is $\aleph_1$-saturated, there exists $\mathbf{c} = (\oc_{A,a})_{A \in K, a \in A}$ in $M$ such that $\mathbf{c} \models \Sigma$.  For each $A \in \KK$, choose $A' \in K$ and an $L_0$-isomorphism $g : A \rightarrow A'$.  Set $f'_A : A \rightarrow M^n$ as $f'_A(a) = \oc_{A',g(a)}$ for all $a \in A$.  Then, check that $(I, f'_A)_{A \in \KK}$ is a $\KK$-configuration into $M^n$.
    \else
    	This follows by compactness and $\kappa$-saturation.
    \fi
\end{proof}

\begin{lemma}\label{Lemma_MakeInjective}
    Let $L_0$ be a relational language and let $\KK$ be a class of finite $L_0$-structures closed under isomorphism.  Let $L$ be a language, $M$ be an infinite $L$-structure, and $n < \omega$.  If there exists a $\KK$-configuration into $M^n$, then there exists an injective $\KK$-configuration into $M^{n+1}$.  In particular, if an $L$-theory with infinite models admits a $\KK$-configuration, then it admits an injective one.
\end{lemma}

\begin{proof}
    Suppose that $(I, f_A)_{A \in \KK}$ is a $\KK$-configuration into $M^n$.  Define the function $I' : \sig(L_0) \rightarrow \Fml(L,M)$ by setting, for each $R \in \sig(L_0)$ of arity $n$,
    \[
        I'(R)(\oy_0, z_0, \oy_1, z_1, ..., \oy_{n-1}, z_{n-1}) = I(R)(\oy_0, \oy_1, ..., \oy_{n-1}).
    \]
    For each $A \in \KK$, define $f'_A : A \rightarrow M^{n+1}$ as follows: First, since $M$ is infinite and $A$ is finite, there exists an injective function $g : A \rightarrow M$.  For each $a \in A$, let $f'_A(a) = (f_A(a), g(a))$.  It is easy to check that $(I', f'_A)_{A \in \KK}$ is an injective $\KK$-configuration into $M^{n+1}$.
    
    Finally, suppose that an $L$-theory $T$ with infinite models admits a $\KK$-configuration.  Then, by Lemma \ref{Lemma_SaturatedConfig}, for any sufficiently saturated (infinite) $M \models T$, there exists a $\KK$-configuration into $M^n$ for some $n < \omega$.  Hence, there exists an injective $\KK$-configuration into $M^{n+1}$.
\end{proof}

We discuss the notion of \emph{reductive subclass} from Definition 4.8 of \cite{GuiPar}, as it provides a general framework for creating new configurations from old ones.

\begin{definition}
    Let $L_0$ and $L_1$ be relational languages such that $\sig(L_0) \subseteq \sig(L_1)$.  For each $t < 2$, let $\KK_t$ be a class of finite $L_t$-structures closed under isomorphism.  We say that $\KK_0$ is a \emph{reductive subclass} of $\KK_1$ if, for all $A \in \KK_0$, there exists $B \in \KK_1$ such that $A \cong_{L_0} B |_{L_0}$.
\end{definition}

\begin{lemma}[Lemma 4.10 of \cite{GuiPar}]\label{Lemma_ReductiveSubclass}
    For each $t < 2$, let $L_t$ be a relational structure and let $\KK_t$ be a class of finite $L_t$-structures closed under isomorphism.  Let $L$ be a language, $M$ and $L$-structure, and $n < \omega$.  If there exists a $\KK_1$-configuration into $M^n$ and $\KK_0$ is a reductive subclass of $\KK_1$, then there exists a $\KK_0$-configuration into $M^n$.

    In particular, if $\KK_0$ is a reductive subclass of $\KK_1$, then $\CC_{\KK_1} \subseteq \CC_{\KK_0}$.
\end{lemma}

\begin{proposition}\label{Proposition_ClassContainment}
	For each $t < 2$, let $L_t$ be a relational language and let $\KK_t$ be a class of finite $L_t$-structures closed under isomorphism.  Suppose that there exists $I : \sig(L_0) \rightarrow \Fml(L_1,\emptyset)$ such that, for all $R \in \sig(L_0)$, $I(R)$ is a (parameter-free) quantifier-free $L_1$-formula.  Moreover, suppose that there exists $n < \omega$ such that, for all $A \in \KK_0$, there exists $B \in \KK_1$ and an injective function $f_A : A \rightarrow B^n$ such that, for all $R \in \sig(L_0)$ and $\oa \in A^{\arity(R)}$,
	\[
		A \models R(\oa) \Longleftrightarrow B \models I(R)(f_A(\oa)).
	\]
	Then,
	\[
		\CC_{\KK_1} \subseteq \CC_{\KK_0}.
	\]
\end{proposition}

\begin{proof}
	\if\papermode0
		Fix $(I, f_A)_{A \in \KK_0}$ as in the hypothesis (say with $n < \omega$).  Let $L$ be a language and $T$ a complete $L$-theory.  Suppose $T \in \CC_{\KK_1}$.  Then, by Lemma \ref{Lemma_MakeInjective}, there exists $M \models T$, $k < \omega$, and $(J, g_B)_{B \in \KK_1}$ an injective $\KK_1$-configuration into $M^k$.  We claim that there is a $\KK_0$-configuration into $M^{kn}$.  For each $R \in \sig(L_0)$, let $I^*(R)$ be the $L$-formula obtained by replacing all instances of $S \in \sig(L_1)$ from $I(R)$ with $J(S)$.  For all $A \in \KK_0$, fix $B \in \KK_1$ such that $f_A(A) \subseteq B^n$ and define $f^*_A : A \rightarrow M^{kn}$ by setting, for all $a \in A$,
		\[
			f^*_A(a) = (g_B((f_A(a))_0), \dots, g_B((f_A(a))_{n-1})).
		\]
		Then, $(I^*, f^*_A)$ is a $\KK_0$-configuration into $M^{kn}$.  Thus, $T \in \CC_{\KK_0}$.
	\else
		This is essentially Proposition 4.13 of \cite{GuiPar}.
	\fi
\end{proof}

Abusing notation, we will call an $(I, f_A)_{A \in \KK_0}$ as in Proposition \ref{Proposition_ClassContainment} a \emph{$\KK_0$-configuration into $\KK_1$}.

\begin{lemma}\label{Lemma_DiagraphGraph}
	If $\DG$ is the class of all finite directed graphs and $\GG$ is the class of all finite graphs, then there exists a $\DG$-configuration into $\GG$.
\end{lemma}

\begin{proof}
	Let $R$ be the binary relation symbol for $\DG$ and let $E$ be the binary relation symbol for $\GG$.  Let
	\[
		I(R)(x_{0,0}, x_{0,1}, x_{1,0}, x_{1,1}) = E(x_{0,0}, x_{1,1}).
	\]
	For any $A \in \DG$, consider the graph on $A \times 2$ given by setting, for all $a, b \in A$,
	\[
		E((a,0), (b,1)) \Longleftrightarrow E((b,1), (a,0)) \Longleftrightarrow A \models R(a,b),
	\]
	and put no other $E$-relations on $A \times 2$.  Let $f_A : A \rightarrow (A \times 2)^2$ be given by, for all $a \in A$,
	\[
		f_A(a) = ((a,0), (a,1)).
	\]
	One can check that $(I, f_A)_{A \in \DG}$ is a $\DG$-configuration into $\GG$.
\end{proof}

\begin{proposition}\label{Proposition_POTTDGGG}
	If $\PO$ is the class of all finite partial orders, $\TT$ is the class of all finite tournaments, $\DG$ is the class of all finite directed graphs, and $\GG$ is the class of all finite graphs, then
	\[
		\CC_{\PO} = \CC_{\TT} = \CC_{\DG} = \CC_{\GG}.
	\]
\end{proposition}

\begin{proof}
	By Lemma \ref{Lemma_DiagraphGraph}, there exists a $\DG$-configuration into $\GG$, hence, by Proposition \ref{Proposition_ClassContainment}, $\CC_{\GG} \subseteq \CC_{\DG}$.  Since $\PO$, $\TT$, and $\GG$ are each reductive subclasses of $\DG$, by Lemma \ref{Lemma_ReductiveSubclass}, $\CC_{\DG} \subseteq \CC_{\PO}, \CC_{\TT}, \CC_{\GG}$.  In light of Proposition \ref{Proposition_ClassContainment}, it suffices to show that there is a $\GG$-configuration into $\PO$ and a $\GG$-configuration into $\TT$.
	
	For the first, let
	\[
		I(E)(x_{0,0}, x_{0,1}, x_{1,0}, x_{1,1}) = x_{0,0} \lhd x_{1,1} \wedge x_{1,0} \lhd x_{0,1}.
	\]
	For any $A \in \GG$, consider the poset $(A \times 2, \lhd)$ given by setting, for all $(a,t), (b,s) \in A \times 2$,
	\[
		(a,t) \lhd (b,s) \Longleftrightarrow t < s \text{ and } A \models E(a,b).
	\]
	Clearly $A \times 2 \in \PO$.  Finally, let $f_A : A \rightarrow (A \times 2)^2$ be given by, for all $a \in A$,
	\[
		f_A(a) = ((a,0), (a,1)).
	\]
	One can check that $(I, f_A)_{A \in \GG}$ is a $\GG$-configuration into $\PO$.
	
	For the second, use $R$ as the binary relation symbol for $\TT$.  Let
	\[
		I(E)(x_{0,0}, x_{0,1}, x_{1,0}, x_{1,1}) = R(x_{0,0}, x_{1,1}) \wedge R(x_{1,0}, x_{0,1}).
	\]
	For any $A \in \GG$, consider the tournament $(A \times 2, R)$ given by setting, for all $(a,t), (b,s) \in A \times 2$,
	\begin{align*}
		R((a,t), (b,s)) \Longleftrightarrow \ & t < s \text{ and } a = b; \text{ or } \\ & t < s \text{ and } A \models E(a,b); \text{ or } \\ & s < t \text{ and } A \models \neg E(a,b).
	\end{align*}
	Clearly $A \times 2 \in \TT$.  Finally, let $f_A : A \rightarrow (A \times 2)^2$ be given by, for all $a \in A$,
	\[
		f_A(a) = ((a,0), (a,1)).
	\]
	One can check that $(I, f_A)_{A \in \GG}$ is a $\GG$-configuration into $\TT$.
	
	By Proposition \ref{Proposition_ClassContainment}, we conclude that
	\[
		\CC_{\PO} = \CC_{\TT} = \CC_{\DG} = \CC_{\GG}.
	\]
\end{proof}

\begin{example}
    The above proof shows that the converse of Lemma \ref{Lemma_ReductiveSubclass} fails.  Namely, note that $\CC_{\PO} = \CC_{\TT}$, even though $\PO$ is not a reductive subclass of $\TT$ nor vice versa.
\end{example}

\subsection{Configuration Classes of Products}

Next, we turn our attention to the interaction between configurations and the \wreathproduct.

\begin{proposition}\label{Proposition_ConfigurationWreath}
    For each $t < 2$, let $L_t$ be a relational language and let $\KK_t$ be a finite class of $L_t$-structures closed under isomorphism.  Let $L$ be a language, $M$ an $L$-structure, and $n_0, n_1 < \omega$.  If there exists a $\KK_0$-configuration into $M^{n_0}$ and an injective $\KK_1$-configuration into $M^{n_1}$, then there exists a $(\KK_0 \wreath \KK_1)$-configuration into $M^{n_0 + n_1}$.
\end{proposition}

\begin{proof}
    Let $(I_0, f_{A,0})_{A \in \KK_0}$ be a $\KK_0$-configuration and let $(I_1, f_{A,1})_{A \in \KK_1}$ be an injective $\KK_1$-configuration.  Let $L_2$ be the language of $\KK_0 \wreath \KK_1$ and define $I : \sig(L_2) \rightarrow \Fml(L,M)$ as follows:
    \begin{enumerate}
        \item if $R \in \sig(L_0)$ of arity $n$, let
        \begin{align*}
        	I(R) & (\oy_0, \oz_0, \oy_1, \oz_1, ..., \oy_{n-1}, \oz_{n-1}) = \\ & I_0(R)(\oy_0, \oy_1, ..., \oy_{n-1}) \wedge \bigwedge_{i, j < n} \oz_i = \oz_j.
        \end{align*}
        \item if $R \in \sig(L_1)$ of arity $n$, let
        \[
        	I(R)(\oy_0, \oz_0, \oy_1, \oz_1, ..., \oy_{n-1}, \oz_{n-1}) = I_1(R)(\oz_0, \oz_1, ..., \oz_{n-1}).
        \]
        \item $I(E)(\oy_0, \oz_0, \oy_1, \oz_1) = [\oz_0 = \oz_1]$.
    \end{enumerate}
    Fix $B \in \KK_1$ and $A_b \in \KK_0$ for each $b \in B$ and let $C = \bigsqcup_{b \in B} A_b$.  We define $f_C : C \rightarrow M^{n_0 + n_1}$ as follows: For all $b \in B$ and $a \in A_b$, let
    \[
        f_C(a, b) = (f_{A_b,0}(a), f_{B,1}(b)).
    \]
    We show that $(I, f_C)_{C \in \KK_0 \wreath \KK_1}$ is a $(\KK_0 \wreath \KK_1)$-configuration into $M^{n_0 + n_1}$.
    
    Fix $B \in \KK_1$ and $A_b \in \KK_0$ for each $b \in B$ and let $C = \bigsqcup_{b \in B} A_b$.  Fix $R \in \sig(L_2)$ of arity $n$, $\ob \in B^n$, and $\oa$ such that $a_i \in A_{b_i}$ for all $i < n$.  If $R \in \sig(L_0)$, then
    \begin{align*}
        & A \models R((a_0, b_0), \dots, (a_{n-1}, b_{n-1})) \Longleftrightarrow \\ & b_i = b_j \text{ for all $i,j$ and } A_{b_0} \models R(\oa) \Longleftrightarrow \\ & f_{B,1}(b_i) = f_{B,1}(b_j) \text{ for all $i,j$ and } M \models I_0(R)(f_{A_{b_0},0}(\oa)) \Longleftrightarrow \\ & M \models I(R)(f_C(a_0, b_0), \dots, f_C(a_{n-1}, b_{n-1})).
    \end{align*}
    (Note that we are relying here on the fact that $f_{B,1}$ is injective.)  If $R \in \sig(L_1)$, then
    \begin{align*}
        & A \models R((a_0, b_0), \dots, (a_{n-1}, b_{n-1})) \Longleftrightarrow \\ & B \models R(\ob) \Longleftrightarrow M \models I_1(R)(f_{B,1}(\ob)) \Longleftrightarrow \\ & M \models I(R)(f_C(a_0, b_0), \dots, f_C(a_{n-1}, b_{n-1})).
    \end{align*}
    Finally, if $R = E$ (and $n = 2$), then
    \begin{align*}
        A \models E((a_0, b_0), (a_1, b_1)) & \Longleftrightarrow b_0 = b_1 \Longleftrightarrow f_{B,1}(b_0) = f_{B,1}(b_1) \\ &  \Longleftrightarrow M \models I(E)(f_C(a_0, b_0), f_C(a_1, b_1)).
    \end{align*}
    (Again, the relies on the fact that $f_{B,1}$ is injective.)  This concludes the proof.
\end{proof}

Note that, in the previous proposition, we need the assumption of injectivity on the $\KK_1$-configuration to get a $(\KK_0 \wreath \KK_1)$-configuration into $M^{n_0 + n_1}$.  If $M$ is infinite, then we can drop that assumption and use Lemma \ref{Lemma_MakeInjective} to get a $(\KK_0 \wreath \KK_1)$-configuration into $M^{n_0 + n_1 + 1}$.

\begin{proposition}\label{Proposition_ConfigurationWreathConv}
    For each $t < 2$, let $L_t$ be a relational language and let $\KK_t$ be a class of finite $L_t$-structures closed under isomorphism that contains a singleton structure.  Let $L$ be a language, $M$ an $L$-structure, and $n < \omega$.  If there exists a $(\KK_0 \wreath \KK_1)$-configuration into $M^n$, then there exists a $\KK_0$-configuration into $M^n$ and a $\KK_1$-configuration into $M^n$.
\end{proposition}

\begin{proof}
    Since $\KK_1$ contains a singleton structure, $\KK_0$ is a reductive subclass of $\KK_0 \wreath \KK_1$.  By Lemma \ref{Lemma_ReductiveSubclass}, there exists a $\KK_0$-configuration into $M^n$.  Similarly, there exists a $\KK_1$-configuration into $M^n$.
\end{proof}

\begin{corollary}\label{Corollary_ClassWreath}
    For each $t < 2$, let $L_t$ be a relational language and let $\KK_t$ be a class of finite $L_t$-structures closed under isomorphism that contains a singleton structure.  Then,
    \[
        \CC_{\KK_0} \cap \CC_{\KK_1} = \CC_{\KK_0 \wreath \KK_1}.
    \]
\end{corollary}

\begin{proof}
	\if\papermode0
	    If $T \in \CC_{\KK_0} \cap \CC_{\KK_1}$, then $T$ admits a $\KK_t$-configuration for each $t < 2$.  By Lemma \ref{Lemma_MakeInjective}, $T$ admits an injective $\KK_1$-configuration.  By Proposition \ref{Proposition_ConfigurationWreath}, $T$ admits a $(\KK_0 \wreath \KK_1)$-configuration.  Hence, $T \in \CC_{\KK_0 \wreath \KK_1}$.  Conversely, if $T \in \CC_{\KK_0 \wreath \KK_1}$, then $T$ admits a $(\KK_0 \wreath \KK_1)$-configuration.  By Proposition \ref{Proposition_ConfigurationWreathConv}, $T$ admits a $\KK_t$-configuration for each $t < 2$.  Thus, $T \in \CC_{\KK_0} \cap \CC_{\KK_1}$.
	\else
		Follows from Lemma \ref{Lemma_MakeInjective}, Proposition \ref{Proposition_ConfigurationWreath}, and Proposition \ref{Proposition_ConfigurationWreathConv}.
	\fi
\end{proof}

\begin{example}
    For this anaylsis, it is important that $\CC_{\KK}$ only contains complete theories \emph{with infinite models}.  For example, suppose that $\KK$ is the class of all finite sets with no structure.  Then, $\KK \wreath \KK$ is the class of all finite equivalence relations, $\EE$.  Any consistent theory admits a $\KK$-configuration, but only a theory with an infinite model admits an $\EE$-configuration.
    \if\papermode0
        
        To see this, suppose that $L$ is a language and $(I, f_A)_{A \in \EE}$ be a $\EE$-configuration into $M^n$ for some $L$-structure $M$ and some $n < \omega$.  For any $m < \omega$, let $A$ be the structure that contains $m$ elements, $a_0, ..., a_{m-1}$, that are pairwise $E$-inequivalent.  Then, for all $i, j < m$, $A \models E(a_i, a_j)$ if and only if $i = j$.  Thus,
        \[
            M \models I(E)(f_A(a_i), f_A(a_j)) \Longleftrightarrow i = j.
        \]
        Thus, for all $i, j < m$, if $f_A(a_i) = f_A(a_j)$, then $M \models I(E)(f_A(a_i), f_A(a_j))$, so $i = j$.  That is, $f_A$ is injective.  Therefore, $|M^n| \ge m$.  Since $m$ was arbitrary, $|M^n| \ge \aleph_0$.  Therefore, $M$ is infinite.
    \fi
\end{example}

Next, we analyze the the interaction between configurations and the \directproduct.

\begin{proposition}\label{Proposition_ConfigurationDirect}
    For each $t < 2$, let $L_t$ be a relational language and let $\KK_t$ be a class of finite $L_t$-structures closed under isomorphism.  Let $L$ be a language, $M$ an $L$-structure, and $n_0, n_1 < \omega$.  If there exists an injective $\KK_0$-configuration into $M^{n_0}$ and an injective $\KK_1$-configuration into $M^{n_1}$, then there exists a $(\KK_0 \btimes \KK_1)$-configuration into $M^{n_0 + n_1}$.
\end{proposition}

\begin{proof}
    For each $t < 2$, let $(I_t, f_{A,t})_{A \in \KK_t}$ be a $\KK_t$-configuration into $M^{n_t}$.  Let $L_2$ be the language of $\KK_0 \btimes \KK_1$ and define $I : \sig(L_2) \rightarrow \Fml(L,M)$ as follows:
    \begin{enumerate}
        \item for $t < 2$ and $R \in \sig(L_t)$ of arity $n$, let
            \[
                I(R)(\oy_{0,0}, \oy_{0,1}, \oy_{1,0}, \oy_{1,1}, ..., \oy_{n-1,0}, \oy_{n-1,1}) = I_t(R)(\oy_{0,t}, \oy_{1,t}, ..., \oy_{n-1,t}).
            \]
        \item for $t < 2$, let
            \[
                I(E_t)(\oy_{0,0}, \oy_{0,1}, \oy_{1,0}, \oy_{1,1}) = \left[ \oy_{0,t} = \oy_{1,t} \right].
            \]
    \end{enumerate}
    For each $t < 2$, fix $A_t \in \KK_t$ and let $A \subseteq A_0 \btimes A_1$.  Define $f_A : A \rightarrow M^{n_0 + n_1}$ as follows: For all $(a_0, a_1) \in A$, let
    \[
        f_A(a_0, a_1) = (f_{A_0}(a_0), f_{A_1}(a_1)).
    \]
	\if\papermode0
	    We show that $(I, f_A)_{A \in \KK_0 \btimes \KK_1}$ is a $(\KK_0 \btimes \KK_1)$-configuration into $M^{n_0 + n_1}$.
	
	    Fix $A, A_0, A_1$ as above.  Fix $t < 2$, fix $R \in \sig(L_t)$ of arity $n$, fix $a_{i,j} \in A_j$ for $i < n$ and $j < 2$.  Then,
	    \begin{align*}
			& A \models R((a_{0,0}, a_{0,1}), ..., (a_{n-1,0}, a_{n-1,1})) \Longleftrightarrow \\ & A_t \models R(a_{0,t}, ..., a_{n-1,t}) \Longleftrightarrow \\ & M \models I_t(R)(f_{A_t,t}(a_{0,t}), ..., f_{A_t,t}(a_{n-1,t})) \Longleftrightarrow \\ & M \models I(R)(f_A(a_{0,0}, a_{0,1}), ..., f_A(a_{n-1,0}, a_{n-1,1})).
		\end{align*}
		Similarly, for $t < 2$ and $a_{i,j} \in A_j$ for $i < 2$ and $j < 2$,
		\begin{align*}
			& A \models E_t((a_{0,0}, a_{0,1}), (a_{0,0}, a_{0,1})) \Longleftrightarrow \\ & a_{0,t} = a_{1,t} \Longleftrightarrow f_{A_t,t}(a_{0,t}) = f_{A_t,t}(a_{1,t}) \Longleftrightarrow \\ & M \models I(E_t)(f_A(a_{0,0}, a_{0,1}), f_A(a_{1,0}, a_{1,1})).
		\end{align*}
    	(Here we are using that $f_{A_t,t}$ is injective.)  This concludes the proof.
    \else
    	Similarly to Proposition \ref{Proposition_ConfigurationWreath}, we see that $(I, f_A)_{A \in \KK_0 \btimes \KK_1}$ is a $(\KK_0 \btimes \KK_1)$-configuration into $M^{n_0 + n_1}$.
    \fi
\end{proof}

\begin{proposition}\label{Proposition_ConfigurationDirectConv}
    For each $t < 2$, let $L_t$ be a relational language and let $\KK_t$ be a class of finite $L_t$-structures closed under isomorphism that contains a non-empty structure.  Let $L$ be a language, $M$ an $L$-structure, and $n < \omega$.  If there exists a $(\KK_0 \btimes \KK_1)$-configuration into $M^n$, then there exists a $\KK_0$-configuration into $M^n$ and a $\KK_1$-configuration into $M^n$.
\end{proposition}

\begin{proof}
	Since $\KK_1$ contains a non-empty structure, $\KK_0$ is a reductive subclass of $\KK_0 \btimes \KK_1$.  By Lemma \ref{Lemma_ReductiveSubclass}, there exists a $\KK_0$-configuration into $M^n$.  Similarly, there exists a $\KK_1$-configuration into $M^n$.
\end{proof}

\begin{corollary}\label{Corollary_ClassDirect}
    For each $t < 2$, let $L_t$ be a relational language and let $\KK_t$ be a class of finite $L_t$-structures closed under isomorphism that contains a non-empty structure.  Then,
    \[
        \CC_{\KK_0} \cap \CC_{\KK_1} = \CC_{\KK_0 \btimes \KK_1}.
    \]
\end{corollary}

\begin{proof}
	Follows from Lemma \ref{Lemma_MakeInjective}, Proposition \ref{Proposition_ConfigurationDirect}, and Proposition \ref{Proposition_ConfigurationDirectConv}.
\end{proof}

Finally, we look at the relationship between configurations and the free superposition.  This was examined in \cite{GuiPar}, though in a less general setting.  Nevertheless, the proofs there easily generalize to our current framework.

\begin{proposition}\label{Proposition_ConfigurationFree}
    For each $t < 2$, let $L_t$ be a relational language and let $\KK_t$ be a class of finite $L_t$-structures closed under isomorphism.  Let $L$ be a language, $M$ an $L$-structure, and $n_0, n_1 < \omega$.  If there exists a $\KK_0$-configuration into $M^{n_0}$ and a $\KK_1$-configuration into $M^{n_1}$, then there exists a $(\KK_0 * \KK_1)$-configuration into $M^{n_0 + n_1}$.
\end{proposition}

\begin{proof}
	This is essentially Proposition 4.11 of \cite{GuiPar}.
\end{proof}

\begin{definition}
    For each $t < 2$, let $L_t$ be a relational language and let $\KK_t$ be a class of finite $L_t$-structures closed under isomorphism.  We say that $\KK_0$ and $\KK_1$ have \emph{comparable cardinalities} if, for all $t < 2$ and for all $A \in \KK_t$, there exists $B \in \KK_{1-t}$ such that $|A| = |B|$.
\end{definition}

\begin{proposition}\label{Proposition_ConfigurationFreeConv}
    For each $t < 2$, let $L_t$ be a relational language and let $\KK_t$ be a class of finite $L_t$-structures closed under isomorphism.  Assume $\KK_0$ and $\KK_1$ have comparable cardinalities.  Let $L$ be a language, $M$ an $L$-structure, and $n < \omega$.  If there exists a $(\KK_0 * \KK_1)$-configuration into $M^n$, then there exists a $\KK_0$-configuration into $M^n$ and a $\KK_1$-configuration into $M^n$.
\end{proposition}

\begin{proof}
    Since $\KK_0$ and $\KK_1$ have comparable cardinalities, $\KK_0$ is a reductive subclass of $\KK_0 * \KK_1$.  By Lemma \ref{Lemma_ReductiveSubclass}, there exists a $\KK_0$-configuration into $M^n$.  Similarly, there exists a $\KK_1$-configuration into $M^n$.
\end{proof}

\begin{corollary}\label{Corollary_ClassFree}
    For each $t < 2$, let $L_t$ be a relational language and let $\KK_t$ be a class of finite $L_t$-structures closed under isomorphism.  If $\KK_0$ and $\KK_1$ have comparable cardinalities, then
    \[
        \CC_{\KK_0} \cap \CC_{\KK_1} = \CC_{\KK_0 * \KK_1}.
    \]
\end{corollary}

We can stitch together the results from this section to get the following theorem, which basically says that all products act in the same manner with respect to configurations.

\begin{theorem}\label{Theorem_Classes}
	For each $t < 2$, let $L_t$ be a relational language and let $\KK_t$ be a class of finite $L_t$-structures closed under isomorphism that contains a singleton structure.  If $\KK_0$ and $\KK_1$ have comparable cardinalities, then
	\[
		\CC_{\KK_0 \btimes \KK_1} = \CC_{\KK_0 \wreath \KK_1} = \CC_{\KK_0 * \KK_1} = \CC_{\KK_0} \cap \CC_{\KK_1}.
	\]
\end{theorem}

\begin{proof}
	Follows from Corollary \ref{Corollary_ClassWreath}, Corollary \ref{Corollary_ClassDirect}, and Corollary \ref{Corollary_ClassFree}.
\end{proof}

Theorem \ref{Theorem_Classes} says that using products will not help answer Question \ref{Question_Classes}.

\begin{corollary}
    Let $\mathbf{P}_0$ be a subclass of $\mathbf{P}$ as defined in Question \ref{Question_Classes} that is linearly ordered under inclusion.  Then, $\mathbf{P}_0$ is ``closed'' under taking lexicographic products, full products, and free superposition.
\end{corollary}

\begin{example}
	For example,
	\[
		\CC_{\LO \btimes \GG} = \CC_{\LO \wreath \GG} = \CC_{\LO * \GG} = \CC_{\LO} \cap \CC_{\GG} = \CC_{\GG}.
	\]
\end{example}

\section{Conclusion}\label{Section_Conclusion}

Let $\KK_0$ and $\KK_1$ be classes of finite structures in a relational language that are closed under isomorphism.  The following table summarizes the known results on the preservation of properties under products of classes of structures.  Where $\KK_0$ and $\KK_1$ having a property implies that product also has the property (sometimes under mild assumptions on $\KK_0$ and $\KK_1$), we will write ``yes''.  The two results on 2-amalgamation and the free superposition were known prior to this paper.

\vspace{.1in}

\begin{center}
\setlength{\tabcolsep}{6pt}
\renewcommand{\arraystretch}{1.5}
	\begin{tabular}{| l | c | c | c |}
		 \hline & $\KK_0 \wreath \KK_1$ & $\KK_0 \btimes \KK_1$ & $\KK_0 * \KK_1$ \\ \hline \hline
		\AgeIndivisibility & \hyperref[Theorem_WreathAgeInd]{Yes} & \hyperref[Theorem_DirectAgeInd]{Yes} & \hyperref[Question_FreeAgeInd]{Open} \\ \hline
		Definable Self-Similarity & \hyperref[Example_NotDefSelfSim]{No} & \hyperref[Example_NotDefSelfSim]{No} & \hyperref[Proposition_FreeDefSelfSim]{Yes} \\ \hline
		Amalgamation Property & \hyperref[Proposition_WreathDirectAP]{Yes} & \hyperref[Proposition_WreathDirectAP]{Yes} & No \\ \hline
		Strong Amalgamation & \hyperref[Theorem_WreathStrongAmalg]{Yes} & \hyperref[Example_Direct2Amalg]{No} & Yes \\ \hline
		Disjoint $n$-Amalgamation, $n \ge 3$ & \hyperref[Example_WreathDirectnAmalg]{No} & \hyperref[Example_WreathDirectnAmalg]{No} & \hyperref[Theorem_FreenAmalgamation]{Yes} \\ \hline
	\end{tabular}
\end{center}

\vspace{.1in}

Although we have answered some of the questions about the interaction between properties of classes of structures and products on classes of structures, there are a few questions that remain open.  For example, in Question \ref{Question_FreeAgeInd}, we asked if \ageindivisibility\ is preserved under the free superposition.  It is known that, assuming the hereditary property, being definably self-similar is preserved under the free superposition (Proposition \ref{Proposition_FreeDefSelfSim}).  Moreover, when working with \fraisse~classes with strong amalgamation, being definably self-similar implies being \ageindivisible\ (Lemma \ref{Lemma_DefSelfSimAgeInd}).  Therefore, it seems possible that \ageindivisibility\ is preserved under free superposition, at least for \fraisse~classes.  On the other hand, indivisibility of the \fraisse~limit is not preserved under free superposition (Example \ref{Example_FreeIndivis}).

Addressing Question \ref{Question_Classes}, whether there are additional dividing lines of the form $\CC_{\KK}$, would be of interest, from the standpoint of finding model-theoretic dividing lines.  Evidently, using \wreathproduct s, \directproduct s, or free superpositions does not help answer this question with our current knowledge.  However, if there were classes $\KK_0$ and $\KK_1$ such that $\CC_{\KK_0}$ and $\CC_{\KK_1}$ were incomparable, then $\CC_{\KK_0 * \KK_1}$ would produce a class distinct from $\CC_{\KK_0}$ and $\CC_{\KK_1}$.  Moreover, we are only analyzing the effect of products on configurations at the level of theories.  It may be of interest to develop a $\KK$-rank like what was done in \cite{GuiPar} using another product in place of free superposition.  Would such a product produce a rank that behaves more ``rank-like'' (e.g., has additivity, as in Open Question 6.4 of \cite{GuiPar})?  In future work, we hope to explore this avenue.

This paper presents a generalization of configurations from \cite{GuiPar} to general classes of structures (without assuming $\KK$ is a \fraisse~class with the strong amalgamation property).  We thank the anonymous referee for suggesting the following line of inquiry: 

\begin{question}
    Does one obtain any new classes, $\CC_{\KK}$, considering a more general class $\KK$?
\end{question}

\begin{question}
    What are some of the limitations to the dividing lines obtainable as $\CC_{\KK}$ for some $\KK$?
\end{question}

For example, it is shown Remark 4.21 of \cite{GuiHill} that $\CC_{\KK}$ can never be the class of non-simple theories.  What other ``natural'' dividing lines are not captured by this analysis?  We intend to explore these questions in future endeavors.

\section*{Acknowledgements}

The authors thank Rehana Patel for referring us to \cite{sauer2020} for a more general result than Lemma \ref{Lemma_DefSelfSimAgeInd}, which led to finding the exact result in \cite{FrSa00}, for proper attribution.  We thank James Hanson for helpful insights after seeing the results of this paper presented at a seminar.  We thank Nadav Meir for comments and references that improved the presentation of this paper.

The third author thanks Christian Rosendal for his unpublished notes that gave a detailed exposition of the automorphism group of a certain \fraisse~limit of trees as a wreath product and Dugald Macpherson for the suggestion to read \cite{Cam87} for a reference on wreath products.

Finally, we thank the anonymous referee for references and suggestions about reorganization that significantly improved the presentation of this paper.

\end{document}